\documentclass[11pt]{amsart}

\usepackage{amsmath,amssymb,latexsym,soul,cite,mathrsfs}

\usepackage{color,enumitem,graphicx}
\usepackage[colorlinks=true,urlcolor=blue,
citecolor=red,linkcolor=red,linktocpage,pdfpagelabels,
bookmarksnumbered,bookmarksopen]{hyperref}
\usepackage[english]{babel}
\usepackage{esint}

\usepackage[left=2.9cm,right=2.9cm,top=2.8cm,bottom=2.8cm]{geometry}
\usepackage[hyperpageref]{backref}

\usepackage[colorinlistoftodos]{todonotes}
\makeatletter
\providecommand\@dotsep{5}
\def\listtodoname{List of Todos}
\def\listoftodos{\@starttoc{tdo}\listtodoname}
\makeatother

\numberwithin{equation}{section}
\DeclareMathOperator*{\esssup}{ess\,sup}
\DeclareMathOperator*{\essinf}{ess\,inf}

\newcommand{\Om} {\Omega}

\newtheorem{Theorem}{Theorem}[section]
\newtheorem{Lemma}[Theorem]{Lemma}

\newtheorem{Corollary}[Theorem]{Corollary}
\newtheorem{Remark}[Theorem]{Remark}
\newtheorem{Definition}[Theorem]{Definition}

\newcommand\R{\mathbb R}

\begin{document}

\title[Mixed local and nonlocal parabolic problem]
{Weak Harnack inequality for a mixed local and nonlocal parabolic equation}
\author{Prashanta Garain and Juha Kinnunen}

\address[Prashanta Garain ]
{\newline\indent Department of Mathematics
	\newline\indent
Ben-Gurion University of the Negev
	\newline\indent
P.O.B. 653
\newline\indent
Beer Sheva 8410501, Israel
\newline\indent
Email: {\tt pgarain92@gmail.com} }

\address[Juha Kinnunen ]
{\newline\indent Department of Mathematics
\newline\indent
Aalto University
\newline\indent
P.O. Box 11100, FI-00076, Finland
\newline\indent
Email: {\tt juha.k.kinnunen@aalto.fi} }

%\pretolerance
\begin{abstract}
This article proves a weak Harnack inequality with a tail term for sign changing  supersolutions of a mixed local and nonlocal parabolic equation. 
Our argument is purely analytic. It is based on energy estimates and the Moser iteration technique. 
Instead of the parabolic John-Nirenberg lemma, we adopt a lemma of Bombieri to the mixed local and nonlocal parabolic case. 
To this end, we prove an appropriate reverse H\"older inequality and a logarithmic estimate for weak supersolutions. 
\end{abstract}


\maketitle

\section{Introduction}
In this article, we establish a weak Harnack inequality for the mixed local and nonlocal parabolic Laplace equation
%$$
%\partial_t u+\mathcal{L}u(x,t)=\Delta u(x,t)\text{ in }\Om\times(0,T),\leqno{\mathcal{(P)}}
%$$
\begin{equation}\label{Problem}
\partial_t u+\mathcal{L}u(x,t)=\Delta u(x,t)\text{ in }\Om\times(0,T),
\end{equation}
where $T>0$, $\Om\subset\mathbb{R}^N$ with $N\geq 2$ is a bounded domain (i.e. bounded, open and connected set) and $\mathcal{L}$ is an integro-differential operator of the form
\begin{equation}\label{fLap}
\mathcal{L} u(x,t)=\text{P.V.}\,\int_{\mathbb{R}^N}\big(u(x,t)-u(y,t)\big)K(x,y,t)\,dx\,dy\,dt,
\end{equation}
where $\text{P.V.}$ denotes the principal value and $K$ is a symmetric kernel in $x$ and $y$ such that for some  $0<s<1$ and $\Lambda\geq 1$, we have
\begin{equation}\label{kernel}
\frac{\Lambda^{-1}}{|x-y|^{N+2s}}\leq K(x,y,t)\leq\frac{\Lambda}{|x-y|^{N+2s}}
\end{equation}
uniformly in $t\in(0,T)$. 
%Here $\Delta$ denotes the classical Laplace operator. 
If
$$
K(x,y,t)=\frac{C}{|x-y|^{N+2s}},
$$
for some constant $C$, then $\mathcal{L}$ reduces to the well known fractional Laplace operator $(-\Delta)^s$ and \eqref{fLap} is the mixed local and nonlocal fractional heat equation
\begin{equation}\label{mpl}
\partial_t u+(-\Delta)^s u=\Delta u.
\end{equation}
This kind of evolution equations arise in the study of L\'evy processes, image processing etc, see Dipierro-Valdinoci \cite{DV21} and the references therein for more details on the physical interpretation. 

In the elliptic case, Foondun \cite{Fo} have obtained Harnack and local H\"older continuity estimates for the mixed local and nonlocal problem
\begin{equation}\label{mel}
-\Delta u+(-\Delta)^s u=0.
\end{equation}
Chen-Kim-Song-Vondra\v{c}ek in \cite{CKSV} have proved Harnack estimates for \eqref{mel} by a different approach. 
In addition to symmetry results and strong maximum principles, several other qualitative properties of solutions of \eqref{mel} have recently been studied by Biagi-Dipierro-Valdinoci-Vecchi \cite{BSVV2, BSVV1}, Dipierro-Proietti Lippi-Valdinoci \cite{DPV20, DPV21} and Dipierro-Ros-Oton-Serra-Valdinoci \cite{DRXJV20}. 
For a nonlinear version of  \eqref{mel} with the $p$-Laplace equation, Harnack inequality, local H\"older continuity and other regularity results are discussed in Garain-Kinnunen \cite{GK}. 
In the parabolic case Barlow-Bass-Chen-Kassmann \cite{BBCK} have obtained Harnack inequality for \eqref{mpl}. 
Chen-Kumagai \cite{CK} have also proved a Harnack inequality and local H\"older continuity. 
%In the purely local case, regularity results are extensively studied, see Lindqvist \cite{Lind}, Mal\'{y}-Ziemer \cite{Maly} and references therein. 
For the fractional heat equation $\partial_t u+(-\Delta)^s u=0$, a weak Harnack inequality for globally nonnegative solutions is established by Felsinger-Kassmann \cite{Kassweakharnack}, see also Bonforte-Sire-V\'{a}zquez \cite{Vazquez}, Caffarelli-Chan-Vasseur \cite{Cafchanvas}, Chaker-Kassmann \cite{Kassmanchaker}, Kassmann-Schwab \cite{KassmanSch} and Kim \cite{Kim} for related results. 

The main purpose of this article is to provide a weak Harnack inequality for \eqref{Problem} (Theorem \ref{mainthm}). 
To the best of our knowledge, a weak Harnack inequality is unknown even for the prototype equation \eqref{mpl}.
Our main result is stated for sign changing weak supersolutions of \eqref{Problem}.
For sign changing solutions of nonlocal problems, in both the elliptic and parabolic context, an extra quantity referred to as "Tail" or "Parabolic Tail",  generally appears in the Harnack estimates. This phenomenon has been first observed by Kassmann in \cite{KassmanHarnack} for the fractional Laplace equation $(-\Delta)^s u=0$ and further extended by Di Castro-Kuusi-Palatucci \cite{Kuusilocal, Kuusiharnack} and Brasco-Lindgren-Schikorra \cite{BrascoLind, BLS} to the fractional $p$-Laplace equation. For the parabolic nonlocal case, see Str\"omqvist \cite{Martin}, Brasco-Lindgren-Str\"omqvist \cite{BLM}, Banerjee-Garain-Kinnunen \cite{BGK}, Ding-Zhang-Zhou \cite{DZZ} and the references therein. 
For the mixed local and nonlocal elliptic equations, a new tail quantity appears that captures both the local and nonlocal behavior of the mixed operator as observed in \cite{GK}. 
For the mixed local and nonlocal parabolic problem \eqref{Problem}, we introduce an appropriate tail term (Definition \ref{def.tail}), which captures both local and nonlocal behavior of the mixed equation.

In contrast to the probabilistic approach in \cite{BBCK, CK}, we prove the weak Harnack estimate (Theorem \ref{mainthm}) by analytic techniques. More precisely, we employ the approach of Moser \cite{Moser71} that uses a lemma by Bombieri-Giusti \cite{BomGi}, further avoiding the use of technically demanding parabolic John-Nirenberg lemma, see Moser \cite{Moser64} and Fabes-Garofalo \cite{FaGa}. We discuss energy estimates for negative and positive powers of weak supersolutions for \eqref{Problem} (Lemma \ref{inveng} and Lemma \ref{supenergy1}). 
These energy estimates together with the Sobolev inequality and the Moser iteration technique, enable us to estimate the supremum of the negative power of a weak supersolution (Lemma \ref{invlemma}) of \eqref{Problem} and to prove the reverse H\"older inequality (Lemma \ref{revHolderlemma}). 
Finally, a logarithmic estimate (Lemma \ref{Logestimatelemma}) of weak supersolutions of \eqref{Problem} is deduced that allows us to apply the Bombieri lemma (Lemma \ref{Bombieri}) to establish our main result (Theorem \ref{mainthm}).
%In Section $2$, we discuss some preliminary results in our set up and state the main result. In Section $3$, energy estimates are proved and in Section $4$, several supersolution estimates are obtained. Finally, the main result is proved in Section $5$. 

\section{Preliminaries and main results}
We use the following notation throughout.
We denote the positive and negative parts of $a\in\R$ by $a_+=\max\{a,0\}$ and $a_-=\max\{-a,0\}$, respectively. 
The Lebesgue outer measure of a set $S$ is denoted by $|S|$.
The barred integral sign denotes the corresponding integral average. 
We write $C$ to denote a constant which may vary from line to line or even in the same line. If $C$ depends on $r_1,r_2,\ldots, r_k$, we denote $C=C(r_1,r_2,\ldots,r_k)$.

We recall some known results for the fractional Sobolev spaces, see Di Nezza-Palatucci-Valdinoci \cite{Hitchhiker'sguide} for more details.
\begin{Definition}
Let $0<s<1$ and assume that $\Omega\subset\mathbb{R}^N$ is an open and connected subset of $\mathbb R^N$. 
The fractional Sobolev space $W^{s,2}(\Omega)$ is defined by
$$
W^{s,2}(\Omega)=\left\{u\in L^2(\Omega):\frac{|u(x)-u(y)|}{|x-y|^{\frac{N}{2}+s}}\in L^2(\Omega\times \Omega)\right\}
$$
and it is endowed with the norm
$$
\|u\|_{W^{s,2}(\Omega)}=\left(\int_{\Omega}|u(x)|^2\,dx+\int_{\Omega}\int_{\Omega}\frac{|u(x)-u(y)|^2}{|x-y|^{N+2s}}\,dx\,dy\right)^\frac{1}{2}.
$$
{The fractional Sobolev space with zero boundary values is defined by}
$$
W_{0}^{s,2}(\Omega)={\left\{u\in W^{s,2}(\mathbb{R}^N):u=0\text{ in }\mathbb{R}^N\setminus\Omega\right\}}.
$$
\end{Definition}

Both $W^{s,2}(\Omega)$ and $W_{0}^{s,2}(\Omega)$ are reflexive Banach spaces, see \cite{Hitchhiker'sguide}. 
We denote the classical Sobolev space by $W^{1,2}(\Om)$.
The parabolic Sobolev space $L^2(0,T;W^{1,2}(\Omega))$, $T>0$,
consists of measurable functions $u$ on $\Omega\times(0,T)$ such that 
\begin{equation}\label{fpnorm}
||u||_{L^2(0,T;W^{1,2}(\Omega))}=\left(\int_{0}^{T} ||u(\cdot,t)||^2_{W^{1,2}(\Omega)}\,dt\right)^\frac{1}{2}<\infty.
\end{equation}
The space $L^2_{\mathrm{loc}}(0,T;W^{1,2}_{\mathrm{loc}}(\Omega))$ is defined by requiring the conditions above for every $\Omega' \times [t_1,t_2]\Subset \Omega \times (0,T)$. Here $\Omega'\times [t_1,t_2]\Subset \Omega \times (0,T)$ denotes that $\overline{\Omega'}\times [t_1,t_2]$ is a compact subset of $\Omega \times (0,T)$.

The next result asserts that the classical Sobolev space is continuously embedded in the fractional Sobolev space, see \cite[Proposition 2.2]{Hitchhiker'sguide}.
The argument applies an extension property of $\Omega$ so that we can extend functions from $W^{1,2}(\Omega)$ to $W^{1,2}(\R^N)$ and that the extension operator is bounded.

\begin{Lemma}\label{locnon}
Let $\Omega$ be a smooth bounded domain in $\mathbb{R}^N$ and $0<s<1$. 
There exists a positive constant $C=C(\Omega,N,s)$ such that
$$
||u||_{W^{s,2}(\Omega)}\leq C||u||_{W^{1,2}(\Omega)}
$$
for every $u\in W^{1,2}(\Omega)$.
\end{Lemma}

The following result for the fractional Sobolev spaces with zero boundary value follows from \cite[Lemma $2.1$]{Silvaarxiv}.
The main difference compared to Lemma \ref{locnon} is that the result holds for any bounded domain, 
since for the Sobolev spaces with zero boundary value, we always have a zero extension to the complement.

\begin{Lemma}\label{locnon1}
Let $\Omega$ be a bounded domain in $\mathbb{R}^N$ and $0<s<1$. 
There exists a positive constant $C=C(N,s,\Omega)$ such that
\[
\int_{\mathbb{R}^N}\int_{\mathbb{R}^N}\frac{|u(x)-u(y)|^2}{|x-y|^{N+2s}}\,dx\, dy
\leq C\int_{\Omega}|\nabla u|^2\,dx
\]
for every $u\in W_0^{1,2}(\Omega)$.
Here we consider the zero extension of $u$ to the complement of $\Omega$.
\end{Lemma}

The notion of weak supersolutions for \eqref{Problem} is defined as follows.

\begin{Definition}\label{wksoldef}
A function $u\in L^\infty(0,T;L^\infty(\mathbb{R}^N))$ is a weak supersolution of the problem \eqref{Problem}, if $u\in C(0,T;L^2_{\mathrm{loc}}(\Om))\cap L^2(0,T;W^{1,2}_{\mathrm{loc}}(\Om))$ and for every $\Om'\times[t_1,t_2]\Subset\Om\times(0,T)$ and every nonnegative test function $\phi\in W^{1,2}_{\mathrm{loc}}(0,T;L^2(\Om'))\cap L^2_{\mathrm{loc}}(0,T;W^{1,2}_{0}(\Om'))$, we have
\begin{equation}\label{wksol}
\begin{split}
&\int_{\Om'}u(x,t_2)\phi(x,t_2)\,dx-\int_{\Om'}u(x,t_1)\phi(x,t_1)\,dx-\int_{t_1}^{t_2}\int_{\Om'}u(x,t)\partial_t\phi(x,t)\,dx\,dt\\
&\quad+\int_{t_1}^{t_2}\int_{\Om'}\nabla u\nabla\phi\,dx\,dt
+\int_{t_1}^{t_2}\int_{\mathbb{R}^N}\int_{\mathbb{R}^N}\mathcal{A}\big(u(x,y,t)\big)\big(\phi(x,t)-\phi(y,t)\big)\,d\mu\,dt\geq 0,
\end{split}
\end{equation}
where
\[
\mathcal{A}\big(u(x,y,t)\big)=u(x,t)-u(y,t)
\quad\text{and}\quad
d\mu=K(x,y,t)\,dx\,dy.
\]
\end{Definition}

\begin{Remark}\label{defrmk}
By Lemma \ref{locnon} and Lemma \ref{locnon1}, we observe that the Definition \ref{wksoldef} well stated. 
Note that if $u$ is a weak supersolution of \eqref{Problem}, so is $u+c$ for any scalar $c$.
\end{Remark}

\begin{Remark}\label{Molifier}
Below we obtain energy estimates where the test functions depend on the supersolution itself. The admissibility of these test functions can be justified by using the mollification in time defined for $f\in L^1(\Omega\times (0,T))$ by
\begin{equation}\label{mol}
f_h(x,t):=\frac{1}{h}\int_{0}^{t}e^{\frac{s-t}{h}}f(x,s)\,ds.
\end{equation}
See \cite{Verenacontinuity, KLin} for more details on $f_h$.
\end{Remark}

Next, we define the parabolic tail which appears in estimates throughout the article.

\begin{Definition}\label{def.tail}
Let  $x_0\in\mathbb{R}^N$, $t_1,t_2\in(0,T)$ and $r>0$.
The parabolic tail of a weak supersolution $u$ of \eqref{Problem} (Definition \ref{wksoldef}) is
\begin{equation}\label{loctail}
\mathrm{Tail}_{\infty}(u;x_0,r,t_1,t_2)=r^2\esssup_{t_1<t<t_2}\int_{\mathbb{R}^n\setminus  B_r(x_0)}\frac{|u(y,t)|}{|y-x_0|^{N+2s}}\,dy.
\end{equation}
\end{Definition}

Now we are ready to state our main result, which asserts that a weak Harnack inequality holds for weak supersolutions of \eqref{Problem}.

\begin{Theorem}\label{mainthm}
Assume that $u$ is a weak supersolution of \eqref{Problem} such that 
$u\geq 0$ in $B_R(x_0)\times(t_0-r^2,t_0+r^2)\subset\Om\times(0,T)$.
Let $0<r\leq 1$, $r<\frac{R}{2}$ and
$$
T=\Big(\frac{r}{R}\Big)^2\mathrm{Tail}_{\infty}\big(u_-;x_0,R,t_0-r^2,t_0+r^2\big),
$$
where $\mathrm{Tail}_{\infty}$ is defined by \eqref{loctail}. Then for any $0<q<2-\frac{2}{\kappa}$, where $\kappa>2$ is given by \eqref{kappa}, 
there exists a positive constant $C=C(N,s,\Lambda,q)$ such that
\begin{equation}\label{thm1ine}
\left(\fint_{V^-(\frac{r}{2})}(u+T)^q\,dx\, dt\right)^\frac{1}{q}
\leq C\Big(\essinf_{V^+(\frac{r}{2})}\,u+T\Big),
\end{equation}
where
$
V^-\left(\tfrac{r}{2}\right)
=B_\frac{r}{2}(x_0)\times(t_0-r^2,t_0-\tfrac{3}{4}r^2)$ and
$V^+\left(\tfrac{r}{2}\right)
=B_\frac{r}{2}(x_0)\times(t_0+\tfrac{3}{4}r^2,t_0+r^2)$.
\end{Theorem}

\begin{Corollary}\label{mainthmcor}
If $u\geq 0$ in $\mathbb{R}^N\times(t_0-r^2,t_0+r^2)$ in Theorem \ref{mainthm}, then \eqref{thm1ine} reduces to
\begin{equation}\label{cor}
\left(\fint_{V^-(\frac{r}{2})}u^q\,dx\,dt\right)^\frac{1}{q}\leq C\essinf_{V^+(\frac{r}{2})}\,u.
\end{equation}
\end{Corollary}

We state some useful results that are needed to prove our main result (Theorem \ref{mainthm}).
The following inequalities follows from \cite[Lemma $3.3$]{Kassweakharnack}.
\begin{Lemma}\label{auxineq}
Let $a,b>0$ and $\tau_1,\tau_2\geq 0$.
\begin{enumerate}
\item[(i)] For every $\epsilon>1$, there exists a constant $C(\epsilon)=\max\left\{4,\frac{6\epsilon-5}{2}\right\}$ such that
\begin{equation}\label{ine1}
\begin{split}
(b-a)\big(\tau_1^{\epsilon+1}a^{-\epsilon}-\tau_2^{\epsilon+1}b^{-\epsilon}\big)&\geq\frac{\tau_1\tau_2}{\epsilon-1}\left[\left(\frac{b}{\tau_2}\right)^\frac{1-\epsilon}{2}-\left(\frac{a}{\tau_1}\right)^\frac{1-\epsilon}{2}\right]^2\\
&\quad-C(\epsilon)(\tau_2-\tau_1)^2\left[\left(\frac{b}{\tau_2}\right)^{1-\epsilon}+\left(\frac{a}{\tau_1}\right)^{1-\epsilon}\right],
\end{split}
\end{equation}
\item[(ii)] For every $\epsilon\in(0,1)$, there exist constants $\zeta(\epsilon)=\frac{4\epsilon}{1-\epsilon}$, $\zeta_1(\epsilon)=\frac{\zeta(\epsilon)}{6}$ and $\zeta_2(\epsilon)=\zeta(\epsilon)+\frac{9}{\epsilon}$ such that  
\begin{equation}\label{ine2}
\begin{split}
(b-a)\big(\tau_1^{2}a^{-\epsilon}-\tau_2^{2}b^{-\epsilon}\big)
\geq\zeta_1(\epsilon)\big(\tau_2 b^\frac{1-\epsilon}{2}-\tau_1 a^\frac{1-\epsilon}{2}\big)^2-\zeta_2(\epsilon)(\tau_2-\tau_1)^2\big(b^{1-\epsilon}+a^{1-\epsilon}\big).
\end{split}
\end{equation}
\end{enumerate}
\end{Lemma} 

%Next we state the following weighted Poincar\'e inequality that follows from \cite[Corollary $3$]{Kmnine}. {\color{blue} (Should we only state Lemma \ref{wgtPoin} by referring to \cite[Corollary $3$]{Kmnine} in place of Lemma \ref{KmnPoin}?)}
%\begin{Lemma}\label{KmnPoin}
%Let $\phi:B_1(0)\to[0,\infty)$ be a radially decreasing function. 
%There exists a constant $C=C(N,\phi)$ such that
%\begin{equation}\label{Kmnwgtine}
%\fint_{B_1(0)}|u-u_{\phi}|^2\,\phi\,dx\leq C \fint_{B_r(x_0)}|\nabla u|^2\phi\,dx
%\end{equation}
%for every $u\in W^{1,2}(B_1(0))$, where
%$$
%u_{\phi}=\frac{\int_{B_1(0)}u\phi\,dx}{\int_{B_1(0)}\phi\,dx}.
%$$
%\end{Lemma}

Next, we state a weighted Poincar\'e inequality that follows from \cite[Corollary $3$]{Kmnine} by a change of variables.
See also \cite[Theorem 5.3.4]{Scoste}. 
This plays an important role in the logarithmic estimate for supersolutions (Lemma \ref{Logestimatelemma}).

\begin{Lemma}\label{wgtPoin}
Let $\phi:B_r(x_0)\to[0,\infty)$ be a radially decreasing function. 
There exists a positive constant $C=C(N,\phi)$ such that
\begin{equation}\label{wgtine}
\fint_{B_r(x_0)}|u-u_{\phi}|^2\,\phi\,dx\leq C r^2\fint_{B_r(x_0)}|\nabla u|^2\phi\,dx
\end{equation}
for every $u\in W^{1,2}(B_r(x_0))$, where
$$
u_{\phi}=\frac{\int_{B_r(x_0)}u\phi\,dx}{\int_{B_r(x_0)}\phi\,dx}.
$$
\end{Lemma}

The following version of the Gagliardo-Nirenberg-Sobolev inequality will be useful for us,
see \cite[Corollary 1.57]{Maly}.
\begin{Lemma}\label{c.omega_sobo}
Let $\Omega\subset\mathbb{R}^N$ be an open set with $\lvert\Omega\rvert<\infty$ and denote
\begin{equation}\label{kappa}
\kappa=
\begin{cases}
\frac{2N}{N-2},&\quad N>2,\\
4,&\quad N=2.
\end{cases}
\end{equation}
There exists a positive constant $C=C(N)$ such that
\begin{equation}\label{e.friedrich}
\biggl(\int_\Omega \lvert u\rvert^{\kappa }\,dx\biggr)^{\frac{1}{\kappa}}
\le C \lvert\Omega\rvert^{\frac{1}{N}-\frac{1}{2}+\frac{1}{\kappa}} \biggl(\int_\Omega \lvert \nabla u\rvert^{2}\,dx\biggr)^{\frac{1}{2}}
\end{equation}
for every $u\in W_0^{1,2}(\Omega)$.
\end{Lemma}

Our final auxiliary result is a lemma of Bombieri-Giusti \cite{BomGi}, which can be proved with the same arguments as in the proof of \cite[Lemma 2.11]{Kin-Kuusi} and \cite[Lemma 2.2.6]{Scoste}.

\begin{Lemma}\label{Bombieri}
Assume that $\nu$ is a Borel measure on $\R^{N+1}$ and let $\theta$, $A$ and $\gamma$ be positive constants, $0<\delta<1$ and $0<\alpha\leq\infty$. Let $U(\sigma)$ be bounded measurable sets with $U(\sigma')\subset U(\sigma)$ for $0<\delta\leq\sigma'<\sigma\leq 1$. Let $f$ be a positive measurable function on $U(1)$ which satisfies the reverse H\"older inequality
$$
\left(\fint_{U(\sigma')}f^\alpha\,d\nu\right)^\frac{1}{\alpha}
\leq\left(\frac{A}{(\sigma-\sigma')^{\theta}}\fint_{U(\sigma)}f^\beta\,d\nu\right)^\frac{1}{\beta},
$$
with $0<\beta<\min\{1,\alpha\}$. Further assume that $f$ satisfies
$$
\left|\{x\in U(1):\log f>\lambda\}\right|\leq\frac{A|U(\delta)|}{\lambda^\gamma}
$$
for all $\lambda>0$. Then 
$$
\left(\fint_{U(\delta)}f^\alpha\,d\nu\right)^\frac{1}{\alpha}\leq C,
$$
for some constant $C=C(\theta,\delta,\alpha,\gamma,A)>0$.
\end{Lemma}

\section{Energy estimates}
In this section, we establish energy estimates for weak supersolutions of \eqref{Problem}. The first one is the following lemma that helps us to estimate the supremum of the inverse of weak supersolutions.
\begin{Lemma}\label{inveng}
Assume that $u$ is a weak supersolution of \eqref{Problem} such that
$u\geq 0$ in $B_R(x_0)\times(\tau_1-\tau,\tau_2)\subset\Om\times(0,T)$.
Let $0<r\leq 1$ be such that $r<R$ and denote $v=u+l$, $l>0$. Then for any $m>0$ there exists a positive constant $C(m)\approx 1+m$ such that
\begin{equation}\label{negeng1}
\begin{split}
&\int_{\tau_1-\tau}^{\tau_2}\int_{B_r(x_0)}|\nabla v^{-\frac{m}{2}}|^2\psi(x)^{m+2}\eta(t)\,dx\,dt\\
&\leq\frac{m^2C(m+1)}{m+1}\int_{\tau_1-\tau}^{\tau_2}\int_{B_r(x_0)}\big(\psi(x)-\psi(y)\big)^2\left[\left(\frac{v(x,t)}{\psi(x)}\right)^{-m}+\left(\frac{v(y,t)}{\psi(y)}\right)^{-m}\right]\eta(t)\,d\mu\,dt\\
&+\frac{m^2}{m+1}\Bigg[2\Lambda\esssup_{x\in\mathrm{supp}\psi}\int_{\mathbb{R}^N\setminus B_r(x_0)}\frac{dy}{|x-y|^{N+2s}}\int_{\tau_1-\tau}^{\tau_2}\int_{B_r(x_0)}v(x,t)^{-m}\psi(x)^{m+2}\eta(t)\,dx\,dt\\
&+\frac{2\Lambda}{l}\esssup_{\tau_1-\tau<t<\tau_2,\,x\in\mathrm{supp}\psi}\int_{\mathbb{R}^N\setminus B_R(x_0)}\frac{u_-(y,t)}{|x-y|^{N+2s}}\,dy\int_{\tau_1-\tau}^{\tau_2}\int_{B_r(x_0)}v(x,t)^{-m}\psi(x,t)^{m+2}\eta(t)\,dx\,dt\Bigg]\\
&+\frac{m^2(m+2)^2}{(m+1)^2}\int_{\tau_1-\tau}^{\tau_2}\int_{B_r(x_0)}|\nabla\psi|^2 v(x,t)^{-m}\psi(x)^{m}\eta(t)\,dx\,dt\\
&+\frac{m}{m+1}\int_{\tau_1-\tau}^{\tau_2}\int_{B_r(x_0)}v(x,t)^{-m}\psi(x)^{m+2}|\partial_t\eta(t)|\,dx\,dt
\end{split}
\end{equation}
and
\begin{equation}\label{negeng2}
\begin{split}
&\esssup_{\tau_1<t<\tau_2}\int_{B_r(x_0)}v(x,t)^{-m}\psi(x)^{m+2}\,dx\\
&\leq mC(m+1)\int_{\tau_1-\tau}^{\tau_2}\int_{B_r(x_0)}\big(\psi(x)-\psi(y)\big)^2\left[\left(\frac{v(x,t)}{\psi(x)}\right)^{-m}+\left(\frac{v(y,t)}{\psi(y)}\right)^{-m}\right]\eta(t)\,d\mu\,dt\\
&+m\Bigg[2\Lambda\esssup_{x\in\mathrm{supp}\psi}\int_{\mathbb{R}^N\setminus B_r(x_0)}\frac{dy}{|x-y|^{N+2s}}\int_{\tau_1-\tau}^{\tau_2}\int_{B_r(x_0)}v(x,t)^{-m}\psi(x)^{m+2}\eta(t)\,dx\,dt\\
&+\frac{2\Lambda}{l}\esssup_{\tau_1-\tau<t<\tau_2,\,x\in\mathrm{supp}\psi}\int_{\mathbb{R}^N\setminus B_R(x_0)}\frac{u_-(y,t)}{|x-y|^{N+2s}}\,dy\int_{\tau_1-\tau}^{\tau_2}\int_{B_r(x_0)}v(x,t)^{-m}\psi(x,t)^{m+2}\eta(t)\,dx\,dt\Bigg]\\
&+\frac{m(m+2)^2}{(m+1)}\int_{\tau_1-\tau}^{\tau_2}\int_{B_r(x_0)}|\nabla\psi|^2 v(x,t)^{-m}\psi(x)^{m}\eta(t)\,dx\,dt\\
&+\int_{\tau_1-\tau}^{\tau_2}\int_{B_r(x_0)}v(x,t)^{-m}\psi(x)^{m+2}|\partial_t\eta(t)|\,dx\,dt
\end{split}
\end{equation}
hold for every nonnegative $\psi\in C_c^{\infty}(B_r(x_0))$ and $\eta\in C^\infty(\mathbb{R})$  such that $\eta(t)\equiv 0$ if $t\leq\tau_1-\tau$ and $\eta(t)\equiv 1$ if $t\geq\tau_2$. 
\end{Lemma}

\begin{proof}
Let $\epsilon>1$ and $\psi\in C_c^{\infty}(B_r(x_0))$. Let $t_1=\tau_1-\tau,\,t_2\in(\tau_1,\tau_2)$ and $\eta\in C^\infty(t_1,t_2)$ be such that $\eta(t_1)=0$ and $\eta(t)=1$ for all $t\geq t_2$. Note that $v=u+l$ is again a weak supersolution of \eqref{Problem} and we observe that 
$$
\phi(x,t)=v(x,t)^{-\epsilon}\psi(x)^{\epsilon+1}\eta(t)
$$
is an admissible test function in \eqref{wksol}. Indeed, following \cite{KLin, Verenacontinuity}, since $v$ is a weak supersolution of \eqref{Problem}, noting the definition of $(\cdot)_h$ from \eqref{mol}, we have
\begin{equation}\label{smth1}
\lim_{h\to 0}(I_{h}+L_h+A_h)\geq 0,
\end{equation}
where
$$
I_h=\int_{t_1}^{t_2}\int_{B_r(x_0)}\partial_t v_{h}\phi\,dx\,dt,
$$
$$
L_h=\int_{t_1}^{t_2}\int_{B_r(x_0)}\big(\nabla v(x,t)\big)_h\nabla\phi\,dx\,dt
$$
and
$$
A_h=\int_{t_1}^{t_2}\int_{\mathbb{R}^N}\int_{\mathbb{R}^N}\big((v(x,t)-v(y,t)\big)K(x,y)\big)_h\big(\phi(x,t)-\phi(y,t)\big)\,dx\,dy\,dt.
$$
\textbf{Estimate of $I_h$:} We observe that
\begin{align*}
I_h&=\int_{t_1}^{t_2}\int_{B_r(x_0)}\partial_t v_{h} v(x,t)^{-\epsilon}\psi(x)^{\epsilon+1}\eta(t)\,dx\,dt\\
&=\int_{t_1}^{t_2}\int_{B_r(x_0)}\partial_t v_{h}\left(v(x,t)^{-\epsilon}-v_h(x,t)^{-\epsilon}\right)\psi(x)^{\epsilon+1}\eta(t)\,dx\,dt\\
&\quad+\int_{t_1}^{t_2}\int_{B_r(x_0)}\partial_t v_{h}v_h(x,t)^{-\epsilon}\psi(x)^{\epsilon+1}\eta(t)\,dx\,dt.
\end{align*}
Since  
$$
\partial_t v_h=\frac{v-v_h}{h},
$$
we have
$$
\partial_t v_h\big(v^{-\epsilon}-v_h^{-\epsilon}\big)\leq 0,
$$
over the set $B_r(x_0)\times(t_1,t_2)$. Therefore, the first integral in the above estimate of $I_h$ is nonpositive. As a consequence, we obtain
\begin{equation}\label{llim2}
\begin{split}
I_h&\leq\int_{t_1}^{t_2}\int_{B_r(x_0)}\partial_t v_{h}v_h(x,t)^{-\epsilon}\psi(x)^{\epsilon+1}\eta(t)\,dx\,dt\\
&=\frac{1}{1-\epsilon}\int_{t_1}^{t_2}\int_{B_r(x_0)}\partial_t v_h^{1-\epsilon}\psi(x)^{\epsilon+1}\eta(t)\,dx\,dt.
\end{split}
\end{equation}
Now passing the limit as $h\to 0$, we have
\begin{equation}\label{estI0}
\lim_{h\to 0}I_h\leq I_0,
\end{equation}
where
\begin{equation}\label{I0new}
\begin{split}
I_0:=-\frac{1}{\epsilon-1}\int_{B_r(x_0)}\psi(x)^{\epsilon+1}v(x,t_2)^{1-\epsilon}\,dx+\frac{1}{\epsilon-1}\int_{t_1}^{t_2}\int_{B_r(x_0)}\psi(x)^{\epsilon+1}v(x,t)^{1-\epsilon}\partial_t\eta(t)\,dx\,dt.
\end{split}
\end{equation}
\textbf{Estimate of $L_h$:} Since $\nabla v\in L^{2}(B_r(x_0)\times(t_1,t_2))$, by \cite[Lemma $2.2$]{KLin}, we have
\begin{equation}\label{llim}
L:=\lim_{h\to 0}L_h=\int_{t_1}^{t_2}\int_{B_r(x_0)}\nabla v\nabla\phi\,dx\,dt.
\end{equation}
\textbf{Estimate of $A_h$:} Following the same arguments as in the proof of \cite[Lemma $3.1$]{BGK}, we obtain
\begin{equation}\label{nlim}
A:=\lim_{h\to 0}A_h=\int_{t_1}^{t_2}\int_{\mathbb{R}^N}\int_{\mathbb{R}^N}\big(v(x,t)-v(y,t)\big)\big(\phi(x,t)-\phi(y,t)\big)\,d\mu\,dt.
\end{equation}
Therefore, using \eqref{estI0}, \eqref{llim}, \eqref{nlim} in \eqref{smth1}, we have
\begin{equation}\label{supentestneg}
\begin{split}
0&\leq I_0+A+L\\
&=I_0+\int_{t_1}^{t_2}\int_{B_r(x_0)}\int_{B_r(x_0)}\big(v(x,t)-v(y,t)\big)\big(\phi(x,t)-\phi(y,t)\big)\,d\mu\,dt\\
&+2\int_{t_1}^{t_2}\int_{\mathbb{R}^N\setminus B_r(x_0)}\int_{B_r(x_0)}\big(v(x,t)-v(y,t)\big)\phi(x,t)\,d\mu\,dt\\
&+\int_{t_1}^{t_2}\int_{B_r(x_0)}\nabla v\nabla\phi\,dx\,dt\\
&:=I_0+I_1+I_2+I_3.
\end{split}
\end{equation}
\textbf{Estimate of $I_1$:} Using [\eqref{ine1}, Lemma \ref{auxineq}] for $C(\epsilon)=\max\big\{4,\frac{6\epsilon-5}{2}\big\}$, we have 
\begin{equation}\label{estI1}
\begin{split}
I_1&=\int_{t_1}^{t_2}\int_{B_r(x_0)}\int_{B_r(x_0)}\big(v(x,t)-v(y,t)\big)\Big(v(x,t)^{-\epsilon}\psi(x)^{\epsilon+1}-v(y,t)^{-\epsilon}\psi(y)^{\epsilon+1}\Big)\eta(t)\,d\mu\,dt\\
&\leq-\int_{t_1}^{t_2}\int_{B_r(x_0)}\int_{B_r(x_0)}\frac{\psi(x)\psi(y)}{\epsilon-1}\left[\left(\frac{v(x,t)}{\psi(x)}\right)^\frac{1-\epsilon}{2}-\left(\frac{v(y,t)}{\psi(y)}\right)^\frac{1-\epsilon}{2}\right]^2 \eta(t)\,d\mu\,dt\\
&+C(\epsilon)\int_{t_1}^{t_2}\int_{B_r(x_0)}\int_{B_r(x_0)}\big(\psi(x)-\psi(y)\big)^2\left[\left(\frac{v(x,t)}{\psi(x)}\right)^{1-\epsilon}+\left(\frac{v(y,t)}{\psi(y)}\right)^{1-\epsilon}\right]\eta(t)\,d\mu\,dt.
\end{split}
\end{equation}
\textbf{Estimate of $I_2$:} Since $l>0$ and $u\geq 0$ in $B_R(x_0)\times(t_1,t_2)$, we have $v\geq l$ in $B_R(x_0)\times(t_1,t_2)$. Further, for any $x\in B_R(x_0)$, $y\in\mathbb{R}^N$ and $t\in(t_1,t_2)$, we have
$$
v(x,t)-v(y,t)\leq v(x,t)+u_-(y,t)\text{ and } v(x,t)^{-\epsilon}=v(x,t)^{1-\epsilon}v(x,t)^{-1}\leq l^{-1}v(x,t)^{1-\epsilon}.
$$ 
Using these estimates and the fact that $u\geq 0$ in $B_R(x_0)\times(t_1,t_2)$, using that $u(y,t)_-$ vanishes in $B_R(x_0)\times(t_1,t_2)$, we have
\begin{equation}\label{estI2}
\begin{split}
I_2&=2\int_{t_1}^{t_2}\int_{\mathbb{R}^N\setminus B_r(x_0)}\int_{B_r(x_0)}\big(v(x,t)-v(y,t)\big)v(x,t)^{-\epsilon}\psi(x)^{\epsilon+1}\eta(t)\,d\mu\,dt\\
&\leq 2\Lambda\esssup_{x\in\mathrm{supp}\psi}\int_{\mathbb{R}^N\setminus B_r(x_0)}\frac{dy}{|x-y|^{N+2s}}\int_{t_1}^{t_2}\int_{B_r(x_0)}v(x,t)^{1-\epsilon}\psi(x)^{\epsilon+1}\eta(t)\,dx\,dt\\
&+\frac{2\Lambda}{l}\esssup_{t_1<t<t_2,\,x\in\mathrm{supp}\psi}\int_{\mathbb{R}^N\setminus B_R(x_0)}\frac{u_-(y,t)}{|x-y|^{N+2s}}\,dy\int_{t_1}^{t_2}\int_{B_r(x_0)}v(x,t)^{1-\epsilon}\psi(x)^{\epsilon+1}\eta(t)\,dx\,dt.
\end{split}
\end{equation}
\textbf{Estimate of $I_3$:} We have
\begin{equation}\label{estI3pre}
\begin{split}
I_3&=\int_{t_1}^{t_2}\int_{B_r(x_0)}\nabla v\nabla\big(v(x,t)^{-\epsilon}\psi(x)^{\epsilon+1}\big)\eta(t)\,dx\,dt\\
&=-\epsilon\int_{t_1}^{t_2}\int_{B_r(x_0)}v(x,t)^{-\epsilon-1}|\nabla v|^2\psi(x)^{\epsilon+1}\eta(t)\,dx\,dt\\
&+(\epsilon+1)\int_{t_1}^{t_2}\int_{B_r(x_0)}\psi(x)^{\epsilon}v(x,t)^{-\epsilon}\nabla v\nabla\psi\eta(t)\,dx\,dt.
\end{split}
\end{equation}
Using Young's inequality
\begin{equation}\label{Young}
ab\leq\delta a^2+\frac{1}{4\delta}b^2,
\end{equation}
for any $a,b\geq 0$ and $\delta>0$, we have
\begin{equation}\label{Youngneg}
\begin{split}
\psi^{\epsilon}v^{-\epsilon}|\nabla v||\nabla\psi|
&=\big(v^{\frac{-\epsilon-1}{2}}\psi^\frac{\epsilon+1}{2}|\nabla v|\big)\big(|\nabla\psi|v^\frac{1-\epsilon}{2}\psi^\frac{\epsilon-1}{2}\big)\\
&\leq\delta v^{-\epsilon-1}\psi^{\epsilon+1}|\nabla v|^2+\frac{1}{4\delta}|\nabla\psi|^2 v^{1-\epsilon}\psi^{\epsilon-1}.
\end{split}
\end{equation}
Choosing $\delta=\frac{\epsilon}{2(\epsilon+1)}$ and employing the estimate \eqref{Youngneg} in \eqref{estI3pre} we obtain
\begin{equation}\label{estI3}
\begin{split}
I_3&\leq-\frac{\epsilon}{2}\int_{t_1}^{t_2}\int_{B_r(x_0)}v(x,t)^{-\epsilon-1}|\nabla v|^2\psi(x)^{\epsilon+1}\eta(t)\,dx\,dt
+\frac{(\epsilon+1)^2}{2\epsilon}\int_{t_1}^{t_2}\int_{B_r(x_0)}|\nabla\psi|^2 v^{1-\epsilon}\psi^{\epsilon-1}\,dx\,dt\\
&\leq-\frac{2\epsilon}{(\epsilon-1)^2}\int_{t_1}^{t_2}\int_{B_r(x_0)}\big|\nabla v^\frac{1-\epsilon}{2}\big|^2\psi(x)^{\epsilon+1}\eta(t)\,dx\,dt
+\frac{(\epsilon+1)^2}{2\epsilon}\int_{t_1}^{t_2}\int_{B_r(x_0)}|\nabla\psi|^2 v^{1-\epsilon}\psi^{\epsilon-1}\,dx\, dt,
\end{split}
\end{equation}
where we have also used the fact that 
$$
v^{-\epsilon-1}|\nabla v|^2=\frac{4}{(\epsilon-1)^2}\big|\nabla v^\frac{1-\epsilon}{2}\big|^2.
$$ 
By combining the estimates \eqref{I0new}, \eqref{estI1}, \eqref{estI2} and \eqref{estI3} in \eqref{supentestneg}, we obtain
\begin{equation}\label{engnegpre}
\begin{split}
&\int_{t_1}^{t_2}\int_{B_r(x_0)}\big|\nabla v^\frac{1-\epsilon}{2}\big|^2\psi(x)^{\epsilon+1}\eta(t)\,dx\,dt\\
&+\frac{\epsilon-1}{2\epsilon}\int_{t_1}^{t_2}\int_{B_r(x_0)}\int_{B_r(x_0)}\psi(x)\psi(y)\left[\left(\frac{v(x,t)}{\psi(x)}\right)^\frac{1-\epsilon}{2}-\left(\frac{v(y,t)}{\psi(y)}\right)^\frac{1-\epsilon}{2}\right]^2 \eta(t)\,d\mu\, dt\\
&+\frac{\epsilon-1}{2\epsilon}\int_{B_r(x_0)}\psi(x)^{\epsilon+1}v(x,t_2)^{1-\epsilon}\,dx\\
&\leq\frac{C(\epsilon)(\epsilon-1)^2}{2\epsilon}\int_{t_1}^{t_2}\int_{B_r(x_0)}\int_{B_r(x_0)}\big(\psi(x)-\psi(y)\big)^2\left[\left(\frac{v(x,t)}{\psi(x)}\right)^{1-\epsilon}+\left(\frac{v(y,t)}{\psi(y)}\right)^{1-\epsilon}\right]\eta(t)\,d\mu\,dt\\
&+\frac{(\epsilon-1)^2}{2\epsilon}\Bigg[2\Lambda\esssup_{x\in\mathrm{supp}\psi}\int_{\mathbb{R}^N\setminus B_r(x_0)}\frac{dy}{|x-y|^{N+2s}}\int_{t_1}^{t_2}\int_{B_r(x_0)}v(x,t)^{1-\epsilon}\psi(x,t)^{\epsilon+1}\eta(t)\,dx\,dt\\
&+\frac{2\Lambda}{l}\esssup_{t_1<t<t_2,\,x\in\mathrm{supp}\psi}\int_{\mathbb{R}^N\setminus B_R(x_0)}\frac{u_-(y,t)}{|x-y|^{N+2s}}\,dy\int_{t_1}^{t_2}\int_{B_r(x_0)}v(x,t)^{1-\epsilon}\psi(x,t)^{\epsilon+1}\eta(t)\,dx\,dt\Bigg]\\
&+\frac{\big(\epsilon^2-1\big)^2}{4\epsilon^2}\int_{t_1}^{t_2}\int_{B_r(x_0)}|\nabla\psi|^2 v(x,t)^{1-\epsilon}\psi(x)^{\epsilon-1}\,dx\,dt\\
&+\frac{\epsilon-1}{2\epsilon}\int_{t_1}^{t_2}\int_{B_r(x_0)}\psi(x)^{\epsilon+1}v(x,t)^{1-\epsilon}|\partial_t\eta(t)|\,dx\,dt.
\end{split}
\end{equation}
Letting $t_2\to\tau_2$ in \eqref{engnegpre}, we have
\begin{equation}\label{engnegsemi}
\begin{split}
&\int_{t_1}^{\tau_2}\int_{B_r(x_0)}|\nabla v^\frac{1-\epsilon}{2}|^2\psi(x)^{\epsilon+1}\eta(t)\,dx\, dt\\
&\leq\frac{C(\epsilon)(\epsilon-1)^2}{2\epsilon}\int_{t_1}^{t_2}\int_{B_r(x_0)}\int_{B_r(x_0)}\big(\psi(x)-\psi(y)\big)^2\left[\left(\frac{v(x,t)}{\psi(x)}\right)^{1-\epsilon}+\left(\frac{v(y,t)}{\psi(y)}\right)^{1-\epsilon}\right]\eta(t)\,d\mu\,dt\\
&+\frac{(\epsilon-1)^2}{2\epsilon}\Bigg[2\Lambda\esssup_{x\in\mathrm{supp}\psi}\int_{\mathbb{R}^N\setminus B_r(x_0)}\frac{dy}{|x-y|^{N+2s}}\int_{t_1}^{t_2}\int_{B_r(x_0)}v(x,t)^{1-\epsilon}\psi(x,t)^{\epsilon+1}\eta(t)\,dx\,dt\\
&+\frac{2\Lambda}{l}\esssup_{t_1<t<t_2,\,x\in\mathrm{supp}\psi}\int_{\mathbb{R}^N\setminus B_R(x_0)}\frac{u_-(y,t)}{|x-y|^{N+2s}}\,dy\int_{t_1}^{t_2}\int_{B_r(x_0)}v(x,t)^{1-\epsilon}\psi(x,t)^{\epsilon+1}\eta(t)\,dx\,dt\Bigg]\\
&+\frac{\big(\epsilon^2-1\big)^2}{4\epsilon^2}\int_{t_1}^{t_2}\int_{B_r(x_0)}|\nabla\psi|^2 v(x,t)^{1-\epsilon}\psi(x)^{\epsilon-1}\,dx\,dt\\
&+\frac{\epsilon-1}{2\epsilon}\int_{t_1}^{t_2}\int_{B_r(x_0)}\psi(x)^{\epsilon+1}v(x,t)^{1-\epsilon}|\partial_t\eta(t)|\,dx\,dt.
\end{split}
\end{equation}
Now choosing $t_2\in(t_1,\tau_2)$ such that
\begin{equation*}\label{negtime}
\int_{B_r(x_0)}\psi(x)^{\epsilon+1}v(x,t_2)^{1-\epsilon}\,dx\geq\esssup_{t_1<t<\tau_2}\int_{B_r(x_0)}\psi(x)^{\epsilon+1}v(x,t)^{1-\epsilon}\,dx,
\end{equation*}
we obtain from \eqref{engnegpre}
\begin{equation}\label{estnegtime}
\begin{split}
&\esssup_{t_1<t<\tau_2}\int_{B_r(x_0)}\psi(x)^{\epsilon+1}v(x,t)^{1-\epsilon}\,dx\\
&\leq C(\epsilon)(\epsilon-1)\int_{t_1}^{t_2}\int_{B_r(x_0)}\int_{B_r(x_0)}\big(\psi(x)-\psi(y)\big)^2\left[\left(\frac{v(x,t)}{\psi(x)}\right)^{1-\epsilon}+\left(\frac{v(y,t)}{\psi(y)}\right)^{1-\epsilon}\right]\eta(t)\,d\mu\,dt\\
&+(\epsilon-1)\Bigg[2\Lambda\esssup_{x\in\mathrm{supp}\psi}\int_{\mathbb{R}^N\setminus B_r(x_0)}\frac{dy}{|x-y|^{N+2s}}\int_{t_1}^{t_2}\int_{B_r(x_0)}v(x,t)^{1-\epsilon}\psi(x,t)^{\epsilon+1}\eta(t)\,dx\,dt\\
&+\frac{2\Lambda}{l}\esssup_{t_1<t<t_2,\,x\in\mathrm{supp}\psi}\int_{\mathbb{R}^N\setminus B_R(x_0)}\frac{u_-(y,t)}{|x-y|^{N+2s}}\,dy\int_{t_1}^{t_2}\int_{B_r(x_0)}v(x,t)^{1-\epsilon}\psi(x,t)^{\epsilon+1}\eta(t)\,dx\,dt\Bigg]\\
&+\frac{(\epsilon-1)(\epsilon+1)^2}{2\epsilon}\int_{t_1}^{t_2}\int_{B_r(x_0)}|\nabla\psi|^2 v(x,t)^{1-\epsilon}\psi(x)^{\epsilon-1}\,dx\,dt\\
&+\int_{t_1}^{t_2}\int_{B_r(x_0)}\psi(x)^{\epsilon+1}v(x,t)^{1-\epsilon}|\partial_t\eta(t)|\,dx\,dt.
\end{split}
\end{equation}
Since $\epsilon>1$ is arbitrary, choosing $m=\epsilon-1$, from \eqref{engnegsemi} and \eqref{estnegtime} the estimates \eqref{negeng1} and \eqref{negeng2} follows respectively.
\end{proof} 

The following energy estimate is useful to obtain reverse H\"older inequality for weak supersolutions.
\begin{Lemma}\label{supenergy1}
Assume that $u$ is a weak supersolution of \eqref{Problem} such that
$u\geq 0$ in $B_R(x_0)\times(\tau_1,\tau_2+\tau)\subset\Om\times(0,T)$.
Let $0<r\leq 1$ be such that $r<R$ and denote $v=u+l$, $l>0$. 
Then for any $0<\alpha<1$, we have
\begin{equation}\label{reveng1}
\begin{split}
&\int_{\tau_1}^{\tau_2+\tau}\int_{B_r(x_0)}\big|\nabla v^\frac{\alpha}{2}\big|^2\psi^2\eta\,dx\,dt\\
&\leq\frac{\alpha^2}{1-\alpha}\Bigg[\zeta_2(1-\alpha)\int_{\tau_1}^{\tau_2+\tau}\int_{B_r(x_0)}\int_{B_r(x_0)}\big(\psi(x)-\psi(y)\big)^2 \big(v(x,t)^\alpha+v(y,t)^\alpha\big)\eta(t)\,d\mu\,dt\\
&+2\Lambda\esssup_{x\in\mathrm{supp}\psi}\int_{\mathbb{R}^N\setminus B_r(x_0)}\frac{dy}{|x-y|^{N+2s}}\int_{\tau_1}^{\tau_2+\tau}\int_{B_r(x_0)}v(x,t)^{\alpha}\psi(x,t)^{2}\eta(t)\,dx\,dt\\
&+\frac{2\Lambda}{l}\esssup_{t_1<t<t_2,\,x\in\mathrm{supp}\psi}\int_{\mathbb{R}^N\setminus B_R(x_0)}\frac{u_-(y,t)}{|x-y|^{N+2s}}\,dy\int_{\tau_1}^{\tau_2+\tau}\int_{B_r(x_0)}v(x,t)^\alpha\psi(x,t)^{2}\eta(t)\,dx\,dt\\
&+\frac{2}{1-\alpha}\int_{\tau_1}^{\tau_2+\tau}\int_{B_r(x_0)}v^{\alpha}|\nabla\psi|^2\eta\,d x dt+\frac{1}{\alpha}\int_{\tau_1}^{\tau_2+\tau}\int_{B_r(x_0)}v(x,t)^{\alpha}\psi(x)^2|\partial_{t}\eta(t)|\,dx\, dt\Bigg]
\end{split}
\end{equation}
and
\begin{equation}\label{reveng2}
\begin{split}
&\esssup_{\tau_1<t<\tau_2}\int_{B_r(x_0)}v(x,t)^{\alpha}\psi(x)^2\,dx\\
&\leq 2\alpha\Bigg[\zeta_2(1-\alpha)\int_{\tau_1}^{\tau_2+\tau}\int_{B_r(x_0)}\int_{B_r(x_0)}\big(\psi(x)-\psi(y)\big)^2\big(v(x,t)^\alpha+v(y,t)^\alpha\big)\eta(t)\,d\mu\,dt\\
&+2\Lambda\esssup_{x\in\mathrm{supp}\psi}\int_{\mathbb{R}^N\setminus B_r(x_0)}\frac{dy}{|x-y|^{N+2s}}\int_{\tau_1}^{\tau_2+\tau}\int_{B_r(x_0)}v(x,t)^{\alpha}\psi(x,t)^{2}\eta(t)\,dx\,dt\\
&+\frac{2\Lambda}{l}\esssup_{t_1<t<t_2,\,x\in\mathrm{supp}\psi}\int_{\mathbb{R}^N\setminus B_R(x_0)}\frac{u_-(y,t)}{|x-y|^{N+2s}}\,dy\int_{\tau_1}^{\tau_2+\tau}\int_{B_r(x_0)}v(x,t)^2\psi(x,t)^{\alpha}\eta(t)\,dx\,dt\\
&+\frac{2}{1-\alpha}\int_{\tau_1}^{\tau_2+\tau}\int_{B_r(x_0)}v^{\alpha}|\nabla\psi|^2\eta\,d x dt+\frac{1}{\alpha}\int_{\tau_1}^{\tau_2+\tau}\int_{B_r(x_0)}v(x,t)^{\alpha}\psi(x)^2|\partial_{t}\eta(t)|\,dx\, dt\Bigg],
\end{split}
\end{equation}
where $\zeta_2(\alpha)=\zeta(\alpha)+\frac{9}{\alpha}$ for $\zeta(\alpha)=\frac{4\alpha}{1-\alpha}$ and $\psi\in C_{c}^{\infty}(B_r(x_0))$ is nonnegative and $\eta\in C^\infty(\mathbb{R})$ is also nonnegative such that $\eta(t)= 1$ if $\tau_1\leq t\leq\tau_2$ and $\eta(t)= 0$ if $t\geq\tau_2+\tau$.
\end{Lemma}

\begin{proof}
Let $\epsilon\in(0,1)$ and $\psi\in C_{c}^{\infty}(B_r(x_0))$. Assume that $t_1\in(\tau_1,\tau_2)$, $t_2=\tau_2+l$ and $\eta\in C^\infty(t_1,t_2)$ such that $\eta(t)=1$ for all $t\leq t_1$ and $\eta(t_2)=0$. Since $v=u+l$ is again a weak supersolution of \eqref{Problem}, choosing 
$$
\phi(x,t)=v(x,t)^{-\epsilon}\psi(x)^2\eta(t)
$$ 
as a test function in \eqref{wksol} (which is again justified by mollifying in time as in the proof of Lemma \ref{inveng}), we obtain
\begin{equation}\label{posengtest}
\begin{split}
0&\leq J_0+J_1+J_2+J_3,
\end{split}
\end{equation}
where
\begin{equation}\label{estJ0}
J_0=-\frac{1}{1-\epsilon}\int_{B_r(x_0)}v^{1-\epsilon}(x,t_1)\psi(x)^2\,dx
-\frac{1}{1-\epsilon}\int_{t_1}^{t_2}\int_{B_r(x_0)}v(x,t)^{1-\epsilon}\psi(x)^2\partial_{t}\eta(t)\,dx\,dt,
\end{equation}
$$
J_1=\int_{t_1}^{t_2}\int_{B_r(x_0)}\int_{B_r(x_0)}\big(v(x,t)-v(y,t)\big)\big(\phi(x,t)-\phi(y,t)\big)\eta(t)\,d\mu\,dt,
$$
$$
J_2=2\int_{t_1}^{t_2}\int_{\mathbb{R}^N\setminus B_r(x_0)}\int_{B_r(x_0)}\big(v(x,t)-v(y,t)\big)\phi(x,t)\,d\mu\,dt
$$
and
$$
J_3=\int_{t_1}^{t_2}\int_{B_r(x_0)}\nabla v\nabla\phi\,dx\,dt.
$$
\textbf{Estimate of $J_1$:} Using [\eqref{ine2}, Lemma \ref{auxineq}], for $\zeta(\epsilon)=\frac{4\epsilon}{1-\epsilon}$, $\zeta_1(\epsilon)=\frac{\zeta(\epsilon)}{6}$ and $\zeta_2(\epsilon)=\zeta(\epsilon)+\frac{9}{\epsilon}$, we have
\begin{equation}\label{estJ1}
\begin{split}
J_1&=\int_{t_1}^{t_2}\int_{B_r(x_0)}\int_{B_r(x_0)}\big(v(x,t)-v(y,t)\big)\big(v(x,t)^{-\epsilon}\psi(x)^2-v(y,t)^{-\epsilon}\psi(y)^2\big)\eta(t)\,d\mu\,dt\\
&\leq -\zeta_1(\epsilon)\int_{t_1}^{t_2}\int_{B_r(x_0)}\int_{B_r(x_0)}\big|\psi(x)v(x,t)^\frac{1-\epsilon}{2}-\psi(y)v(y,t)^\frac{1-\epsilon}{2}\big|^2\eta(t)\,d\mu\, dt\\
&+\zeta_2(\epsilon)\int_{t_1}^{t_2}\int_{B_r(x_0)}\int_{B_r(x_0)}\big(\psi(x)-\psi(y)\big)^2 \big(v(x,t)^{1-\epsilon}+v(y,t)^{1-\epsilon}\big)\eta(t)\,d\mu\, dt.
\end{split}
\end{equation}
\textbf{Estimate of $J_2$:}
Since $l>0$ and $u\geq 0$ in $B_R(x_0)\times(t_1,t_2)$, we have $v\geq l$ in $B_R(x_0)\times(t_1,t_2)$ and following the same arguments as in the proof of the estimate \eqref{estI2}, we have
\begin{equation}\label{estJ2}
\begin{split}
J_2&=2\int_{t_1}^{t_2}\int_{\mathbb{R}^N\setminus B_r(x_0)}\int_{B_r(x_0)}\big(v(x,t)-v(y,t)\big)v(x,t)^{-\epsilon}\psi(x)^2\eta(t)\,d\mu\,dt\\
&\leq 2\Lambda\esssup_{x\in\mathrm{supp}\psi}\int_{\mathbb{R}^N\setminus B_r(x_0)}\frac{dy}{|x-y|^{N+2s}}\int_{t_1}^{t_2}\int_{B_r(x_0)}v(x,t)^{1-\epsilon}\psi(x,t)^{2}\eta(t)\,dx\,dt\\
&+\frac{2\Lambda}{l}\esssup_{t_1<t<t_2,\,x\in\mathrm{supp}\psi}\int_{\mathbb{R}^N\setminus B_R(x_0)}\frac{u_-(y,t)}{|x-y|^{N+2s}}\,dy\int_{t_1}^{t_2}\int_{B_r(x_0)}v(x,t)^{1-\epsilon}\psi(x,t)^{2}\eta(t)\,dx\,dt.
\end{split}
\end{equation}
\textbf{Estimate of $J_3$:} We observe that
\begin{equation}\label{estJ3pre}
\begin{split}
J_3&=\int_{t_1}^{t_2}\int_{B_r(x_0)}\nabla v\nabla\big(v^{-\epsilon}\psi^2\big)\eta(t)\,dx\,dt\\
&=-\epsilon\int_{t_1}^{t_2}\int_{B_r(x_0)}v^{-\epsilon-1}|\nabla v|^2\psi^2\eta\,dx\,dt
+2\int_{t_1}^{t_2}\int_{B_r(x_0)}\psi\eta v^{-\epsilon}\nabla v\nabla\psi\,dx\,dt.
\end{split}
\end{equation}
Using Young's inequality \eqref{Young} for $\delta=\frac{\epsilon}{4}$, we obtain
\begin{equation}\label{Youngpos}
2\psi v^{-\epsilon}|\nabla v| |\nabla\psi|
=\big(\sqrt{2}\psi|\nabla v| v^\frac{-\epsilon-1}{2}\big)\big(\sqrt{2}v^\frac{1-\epsilon}{2}|\nabla\psi|\big)
\leq\frac{\epsilon}{2}\psi^2|\nabla v|^2 v^{-\epsilon-1}+\frac{2}{\epsilon}v^{1-\epsilon}|\nabla\psi|^2.
\end{equation}
Hence using \eqref{Youngpos} in \eqref{estJ3pre}, we have
\begin{equation}\label{estJ3}
\begin{split}
J_3&\leq-\frac{\epsilon}{2}\int_{t_1}^{t_2}\int_{B_r(x_0)}v^{-\epsilon-1}|\nabla v|^2\psi^2\eta\,dx\,dt
+\frac{2}{\epsilon}\int_{t_1}^{t_2}\int_{B_r(x_0)}v^{1-\epsilon}|\nabla\psi|^2\eta\,dx\,dt\\
&=-\frac{2\epsilon}{(1-\epsilon)^2}\int_{t_1}^{t_2}\int_{B_r(x_0)}\big|\nabla v^\frac{1-\epsilon}{2}\big|^2\psi^2\eta\,dx\,dt
+\frac{2}{\epsilon}\int_{t_1}^{t_2}\int_{B_r(x_0)}v^{1-\epsilon}|\nabla\psi|^2\eta\,dx\,dt,
\end{split}
\end{equation}
where in the final step we have used the fact that $v^{-\epsilon-1}|\nabla v|^2=\frac{4}{(1-\epsilon)^2}\big|\nabla v^\frac{1-\epsilon}{2}\big|^2$.

Combining the estimates \eqref{estJ0}, \eqref{estJ1}, \eqref{estJ2} and \eqref{estJ3} in \eqref{posengtest}, we obtain
\begin{equation}\label{posengpre}
\begin{split}
&\int_{t_1}^{t_2}\int_{B_r(x_0)}|\nabla v^\frac{1-\epsilon}{2}|^2\psi^2\eta\,dx\,dt\\
&+\frac{(1-\epsilon)^2\zeta_1(\epsilon)}{2\epsilon}\int_{t_1}^{t_2}\int_{B_r(x_0)}\int_{B_r(x_0)}\left|\psi(x)v(x,t)^\frac{1-\epsilon}{2}-\psi(y)v(y,t)^\frac{1-\epsilon}{2}\right|^2\eta(t)\,d\mu\,dt\\
&+\frac{1-\epsilon}{2\epsilon}\int_{B_r(x_0)}v(x,t_1)^{1-\epsilon}\psi(x)^2\,dx\,dt\\
&\leq\frac{(1-\epsilon)^2}{2\epsilon}\Bigg[\zeta_2(\epsilon)\int_{t_1}^{t_2}\int_{B_r(x_0)}\int_{B_r(x_0)}\big(\psi(x)-\psi(y)\big)^2 \big(v(x,t)^{1-\epsilon}+v(y,t)^{1-\epsilon}\big)\eta(t)\,d\mu\,dt\\
&+2\Lambda\esssup_{x\in\mathrm{supp}\psi}\int_{\mathbb{R}^N\setminus B_r(x_0)}\frac{dy}{|x-y|^{N+2s}}\int_{t_1}^{t_2}\int_{B_r(x_0)}v(x,t)^{1-\epsilon}\psi(x,t)^{2}\eta(t)\,dx\,dt\\
&+\frac{2\Lambda}{l}\esssup_{t_1<t<t_2,\,x\in\mathrm{supp}\psi}\int_{\mathbb{R}^N\setminus B_R(x_0)}\frac{u_-(y,t)}{|x-y|^{N+2s}}\,dy\int_{t_1}^{t_2}\int_{B_r(x_0)}v(x,t)^{1-\epsilon}\psi(x,t)^{2}\eta(t)\,dx\,dt\\
&+\frac{2}{\epsilon}\int_{t_1}^{t_2}\int_{B_r(x_0)}v^{1-\epsilon}|\nabla\psi|^2\eta\,d x dt+\frac{1}{1-\epsilon}\int_{t_1}^{t_2}\int_{B_r(x_0)}v(x,t)^{1-\epsilon}\psi(x)^2|\partial_{t}\eta(t)|\,dx\,dt\Bigg].
\end{split}
\end{equation}
Define $\alpha=1-\epsilon$. Then, letting  $t_1\to\tau_1$ in \eqref{posengpre} we obtain \eqref{reveng1}. Next, choosing $t_1\in(\tau_1,\tau_2)$ such that
\begin{equation*}
\int_{B_r(x_0)}v(x,t_1)^{1-\epsilon}\psi(x)^2\,dx\geq\esssup_{\tau_1<t<\tau_2}\int_{B_r(x_0)}v(x,t)^{1-\epsilon}\psi(x)^2\,dx, 
\end{equation*}
and using \eqref{posengpre}, the estimate \eqref{reveng2} follows.
\end{proof}

\section{Estimates for weak supersolutions}
We establish an estimate of supremum for weak supersolutions of \eqref{Problem}.
\begin{Lemma}\label{invlemma}
 Assume that $u$ is a weak supersolution of \eqref{Problem} such that $u\geq 0$ in $B_R(x_0)\times\big(t_0-r^2,t_0\big)\subset\Om\times(0,T)$. 
 Let $0<r\leq 1$ be such that $r<\frac{R}{2}$.
 Let $d>0$ and $v=u+l$, with 
\begin{equation}\label{l1}
l\geq \Big(\frac{r}{R}\Big)^2\mathrm{Tail}_{\infty}\big(u_-;x_0,R,t_0-r^2,t_0\big)+d,
\end{equation}
where $\mathrm{Tail}_{\infty}$ is defined by \eqref{loctail}. 
For any $0<\beta<1$ and $\kappa$ is given by \eqref{kappa}, there exists constants $C=C(N,s,\Lambda)$ and $\theta=\theta(\kappa)$ such that
\begin{equation}\label{invest}
\esssup_{U^-(\sigma' r)}v^{-1}\leq\left(\frac{C}{(\sigma-\sigma')^\theta}\right)^\frac{1}{\beta}\left(\fint_{U^-(\sigma r)}v^{-\beta}\,dx\,dt\right)^\frac{1}{\beta},
\end{equation}
for every $\frac{1}{2}\leq\sigma'<\sigma\leq 1$, where 
$U^-(\eta r)=B_{\eta r}(x_0)\times(t_0-(\eta r)^2,t_0)$, $\frac{1}{2}\leq\eta\leq 1$.
\end{Lemma}
\begin{proof}
Let us divide the interval $(\sigma',\sigma)$ into $k$ parts by setting
$$
\sigma_0=\sigma,\,\sigma_k=\sigma',\,\sigma_j=\sigma-(\sigma-\sigma')\big(1-\gamma^{-j}\big),
\quad j=1,\dots,k-1,
$$
where $\gamma=2-\frac{2}{\kappa}$, where $\kappa$ is given by \eqref{kappa}. 
Let $r_j=\sigma_j r$, $B_j=B_{r_j}(x_0)$,
$\Gamma_j=(t_0-r_j^2,t_0)$ and $Q_j=U^-(r_j)=B_{j}\times\Gamma_j$.
Let $\psi_j\in C_{c}^\infty(B_j)$ and $\eta_j\in C^\infty(\mathbb{R})$ be such that
$0\leq\psi_j\leq 1$ in $B_j$, $\psi_j\equiv 1$ in $B_{j+1}$,
$0\leq\eta_j\leq 1$ in $\Gamma_j$, $\eta_j(t)=1$ for every $t\geq t_0-r_{j+1}^2$,
$\eta_j(t)=0$ for every $t\leq t_0-r_j^2$,
$\mathrm{dist}\big(\mathrm{supp}\,\psi_j,\mathbb{R}^N\setminus B_j\big)\geq 2^{-j-1}r$,
\[
|\nabla\psi_j|\leq\frac{8\gamma^j}{r(\sigma-\sigma')}\text{ in }B_j
\quad\text{and}\quad
\left|\frac{\partial\eta_j}{\partial t}\right|\leq 8\left(\frac{\gamma^j}{r(\sigma-\sigma')}\right)^2\text{ in }\Gamma_j.
\]
Let $\epsilon\geq 1$ and $m=1+\epsilon$. Then, for $w=v^{-1}$, using Lemma \ref{c.omega_sobo} along with H\"older's inequality with exponents $t=\frac{\kappa}{\kappa-2}$ and $t'=\frac{\kappa}{2}$, for some positive constant $C=C(N)$, we have
\begin{equation}\label{invSob}
\begin{split}
&\fint_{Q_{j+1}}w^{\gamma m}\,dx\,dt
=\fint_{\Gamma_{j+1}}\fint_{B_{j+1}}w^{\gamma m}\,dx\,dt
=\fint_{\Gamma_{j+1}}\fint_{B_{j+1}}w^{\frac{m}{t}+\frac{m\kappa}{2t'}}\,dx\,dt\\
&\leq\fint_{\Gamma_{j+1}}\frac{1}{|B_{j+1}|}\left(\int_{B_{j+1}}w^m\,dx\right)^\frac{1}{t}\left(\int_{B_{j+1}}w^\frac{m\kappa}{2}\,dx\right)^\frac{1}{t'}dt\\
&\leq\fint_{\Gamma_{j+1}}\frac{1}{|B_{j+1}|}\left(\esssup_{\Gamma_{j+1}}\int_{B_{j}}w^m\psi_j^{m+2}\,dx \right)^\frac{1}{t}\left(\int_{B_{j}}w^\frac{m\kappa}{2}\psi_j^\frac{(m+2)\kappa}{2}\eta_j(t)^\frac{\kappa}{2}\,dx\right)^\frac{1}{t'}dt\\
&=\frac{|\Gamma_j||B_j|}{|\Gamma_{j+1}||B_{j+1}|}\left(\esssup_{\Gamma_{j+1}}\fint_{B_{j}}w^m\psi_j^{m+2}\,dx \right)^\frac{1}{t}\fint_{\Gamma_{j}}\left(\fint_{B_{j}}w^\frac{m\kappa}{2}\psi_j^\frac{(m+2)\kappa}{2}\eta_j(t)^\frac{\kappa}{2}\,dx\right)^\frac{1}{t'}dt\\
&\leq C\left(\esssup_{\Gamma_{j+1}}\fint_{B_{j}}w^m\psi_j^{m+2}\,dx\right)^\frac{1}{t}r^2\fint_{\Gamma_j}\fint_{B_j}\left|\nabla\Big(w^\frac{m}{2}\psi_j^\frac{m+2}{2}\eta_j^\frac{1}{2}\Big)\right|^2\,dx\,dt\\
&\leq\frac{C}{r_j^{\gamma N}}\left(\esssup_{\Gamma_{j+1}}\int_{B_{j}}w^m\psi_j^{m+2}\,dx \right)^\frac{1}{t}\\
&\quad\quad\cdot\int_{\Gamma_j}\int_{B_j}\left(\big|\nabla\big(w^\frac{m}{2}\big)\big|^2\psi_j^{(m+2)}\eta_j(t)+m^2 w^m |\nabla\psi_j|^2\right)\,dx\,dt.
\end{split}
\end{equation}
Let
\[
I=\esssup_{\Gamma_{j+1}}\int_{B_{j}}w^m\psi_j^{m+2}\,dx\,dt,\quad 
J=\int_{\Gamma_j}\int_{B_j}\big|\nabla\big(w^\frac{m}{2}\big)\big|^2\psi_j^{(m+2)}\eta_j(t)\,dx\,dt
\]
and
\[
K=m^2\int_{\Gamma_j}\int_{B_j}w^m |\nabla\psi_j|^2\,dx\,dt.
\]
Setting $r=r_j, \tau_1=t_0-r_{j+1}^2$, $\tau_2=t_0$ and $\tau=r_j^2-r_{j+1}^2$ in Lemma \ref{inveng}, for some constant $C=C(\Lambda)>0$, we have
\begin{equation}\label{IJexp}
\begin{split}
I,J&\leq Cm^4\Bigg[\int_{\Gamma_j}\int_{B_j}\big(\psi_j(x)-\psi_j(y)\big)^2\Big[{w(x,t)}^m{\psi_j(x)}^{m}+{w(y,t)}^m{\psi_j(y)}^{m}\Big]\eta_j(t)\,d\mu\,dt\\
&+2\esssup_{x\in\mathrm{supp}\psi_j}\int_{\mathbb{R}^N\setminus B_j}\frac{dy}{|x-y|^{N+2s}}\int_{\Gamma_j}\int_{B_j}w(x,t)^{m}\psi_j(x)^{m+2}\eta_j(t)\,dx\,dt\\
&+\frac{2}{l}\esssup_{t\in\Gamma_j,\,x\in\mathrm{supp}\psi_j}\int_{\mathbb{R}^N\setminus B_R(x_0)}\frac{u_-(y,t)}{|x-y|^{N+2s}}\,dy\int_{\Gamma_j}\int_{B_j}w(x,t)^{m}\psi_j(x,t)^{m+2}\eta_j(t)\,dx\, dt\\
&+\int_{\Gamma_j}\int_{B_j}|\nabla\psi_j|^2 w(x,t)^{m}\psi_j(x)^{m}\eta_j(t)\,dx\,dt
+\int_{\Gamma_j}\int_{B_j}w(x,t)^{m}\psi_j(x)^{m+2}|\partial_t\eta_j(t)|\,dx\,dt\Bigg]\\
&=I_1+I_2+I_3+I_4+I_5.
\end{split}
\end{equation}
\textbf{Estimate of $I_1$:} Using the properties of $\psi_j,\eta_j$ and the fact that $0<r\leq 1$, we have
\begin{equation}\label{estI_1}
\begin{split}
I_1&=Cm^4\int_{\Gamma_j}\int_{B_j}(\psi_j(x)-\psi_j(y))^2\big({w(x,t)}^m{\psi(x)}^{m}+{w(y,t)}^m{\psi_j(y)}^{m}\big)\eta_j(t)\,d\mu\,dt\\
&\leq\frac{C m^4\gamma^{2j}}{r^2(\sigma-\sigma')^2}\esssup_{x\in B_j}\int_{B_j}\frac{|x-y|^2}{|x-y|^{N+2s}}\,dy\int_{\Gamma_j}\int_{B_j}w(x,t)^{m}\,dx\,dt\\
&\leq\frac{C m^4\gamma^{2j}}{r_j^{2s}(\sigma-\sigma')^2}\int_{\Gamma_j}\int_{B_j}w(x,t)^{m}\,dx\,dt\\
&\leq \frac{C m^4\gamma^{2j}}{r_j^{2}(\sigma-\sigma')^2}\int_{\Gamma_j}\int_{B_j}w(x,t)^{m}\,dx\,dt,
\end{split}
\end{equation}
for some constant $C=C(N,s,\Lambda)>0$.\\
\textbf{Estimate of $I_2$:} Without loss of generality, we may assume that $x_0=0$. Then, by noting the properties of $\psi_j$ and $\eta_j$, for any $x\in\mathrm{supp}\,\psi_j$ and $y\in\mathbb{R}^N\setminus B_j$, we have 
$$
\frac{1}{|x-y|}=\frac{1}{|y|}\frac{|y|}{|x-y|}\leq\frac{1}{|y|}\left(1+\frac{|x|}{|x-y|}\right)\leq\frac{1}{|y|}\left(1+\frac{r}{2^{-j-1}r}\right)\leq\frac{2^{j+2}}{|y|}.
$$
This implies
\begin{equation}\label{estI_2}
\begin{split}
I_2&=Cm^4\esssup_{x\in\mathrm{supp}\,\psi_j}\int_{\mathbb{R}^N\setminus B_j}\frac{dy}{|x-y|^{N+2s}}\int_{\Gamma_j}\int_{B_j}w(x,t)^{m}\psi_j(x)^{m+2}\eta_j(t)\,dx\,dt\\
&\leq C2^{N+2s+2}m^4 2^{j(N+2s)}\esssup_{x\in\mathrm{supp}\,\psi_j}\int_{\mathbb{R}^N\setminus B_j}\frac{dy}{|y|^{N+2s}}\int_{\Gamma_j}\int_{B_j}w(x,t)^{m}\,dx\,dt\\
&=\frac{Cm^4 2^{j(N+2s)}}{r_j^{2s}}\int_{\Gamma_j}\int_{B_j}w(x,t)^{m}\,dx\,dt\\
&\leq\frac{Cm^4 2^{j(N+2s)}}{r_j^{2}(\sigma-\sigma')^2}\int_{\Gamma_j}\int_{B_j}w(x,t)^{m}\,dx\,dt,
\end{split}
\end{equation}
for some constant $C=C(N,s,\Lambda)>0$.\\
\textbf{Estimate of $I_3$:} without loss of generality, again we assume that $x_0=0$. Let $x\in\mathrm{supp}\,\psi_j$ and $y\in B_j$, then 
$$
\frac{1}{|x-y|}\leq\frac{1}{|y|}\left(1+\frac{r}{R-r}\right)\leq\frac{2}{|y|}.
$$
By \eqref{l1} we have
\begin{equation}\label{estI_3}
\begin{split}
I_3&=m^4\frac{C}{l}\esssup_{t\in\Gamma_j,\,x\in\mathrm{supp}\psi_j}\int_{\mathbb{R}^N\setminus B_R(x_0)}\frac{u_-(y,t)}{|x-y|^{N+2s}}\,dy\int_{\Gamma_j}\int_{B_j}w(x,t)^{m}\psi_j(x,t)^{m+2}\eta_j(t)\,dx\,dt\\
&\leq\frac{Cm^4}{l}R^{-2}\mathrm{Tail}_\infty\,\big(u_-;0,R,t_0-r^2,t_0\big)\int_{\Gamma_j}\int_{B_j}w(x,t)^{m}\psi_j(x,t)^{m+2}\eta_j(t)\,dx\,dt\\
&\leq\frac{Cm^4}{r^2(\sigma-\sigma')^2}\int_{\Gamma_j}\int_{B_j}w(x,t)^{m}\,dx\,dt,
\end{split}
\end{equation}
for some constant $C=C(N,s,\Lambda)>0$.\\
\textbf{Estimate of $I_4$:} By the properties of $\psi_j$ and $\eta_j$, we obtain
\begin{equation}\label{estI_4}
\begin{split}
I_4&=C m^4\int_{\Gamma_j}\int_{B_j}|\nabla\psi_j|^2 w(x,t)^{m}\psi_j(x)^{m}\eta_j(t)\,dx\,dt\\
&\leq\frac{Cm^4\gamma^{2j}}{r_j^{2}(\sigma-\sigma')^2}\int_{\Gamma_j}\int_{B_j}w(x,t)^{m}\,dx\,dt,
\end{split}
\end{equation}
for some constant $C=C(N,s,\Lambda)>0$.\\
\textbf{Estimate of $I_5$:} By the properties of $\psi_j$ and $\eta_j$, we have
\begin{equation}\label{estI_5}
\begin{split}
I_5&=C m^4\int_{\Gamma_j}\int_{B_j}w(x,t)^{m}\psi_j(x)^{m+2}|\partial_t\eta_j(t)|\,dx\,dt\\
&\leq\frac{Cm^4\gamma^{2j}}{r_j^{2}(\sigma-\sigma')^2}\int_{\Gamma_j}\int_{B_j}w(x,t)^{m}\,dx\,dt,
\end{split}
\end{equation}
for some constant $C=C(N,s,\Lambda)>0$. Plugging the estimates \eqref{estI_1}\, \eqref{estI_2}, \eqref{estI_3}, \eqref{estI_4} and \eqref{estI_5} in \eqref{IJexp}, since $\gamma<2$, we obtain
\begin{equation}\label{estIJ}
\begin{split}
I,J&\leq\frac{Cm^4 2^{j(N+4)}}{r_j^2(\sigma-\sigma')^2}\int_{\Gamma_j}\int_{B_j}w(x,t)^{m}\,dx\,dt,
\end{split}
\end{equation}
for some positive constant $C=C(N,s,\Lambda)$. Again, using the properties of $\psi_j$, we obtain
\begin{equation}\label{estK}
\begin{split}
K&\leq\frac{Cm^4 2^{2j}}{r_j^2(\sigma-\sigma')^2}\int_{\Gamma_j}\int_{B_j}w(x,t)^{m}\,dx\,dt,
\end{split}
\end{equation}
for some positive constant $C=C(N,s,\Lambda)$. Employing the estimates \eqref{estIJ} and \eqref{estK} in \eqref{invSob}, for $m=1+\epsilon$, $\epsilon\geq 1$ and $\gamma=2-\frac{2}{\kappa}$, we have
\begin{equation}\label{invMoser}
\fint_{Q_{j+1}}w^{\gamma m}\,dx\,dt
\leq C\left(\frac{m^4 2^{j(N+4)}}{(\sigma-\sigma')^2}\fint_{Q_j}w^m\,dx\,dt\right)^\gamma,
\end{equation}
for some positive constant $C=C(N,s,\Lambda)$. Now, we use Moser's iteration technique to prove the estimate \eqref{invest}. 
Let $m_j=2\gamma^j$,  $j=0,1,2,\dots$. By iterating \eqref{invMoser}, we have
\begin{equation}\label{sbigp}
\begin{split}
\left(\fint_{Q_0}w^2\,dx\,dt\right)^\frac{1}{2}
&\geq\left(\frac{\sigma-\sigma'}{C}\right)^{1+\gamma^{-1}+\gamma^{-2}+\cdots+\gamma^{1-k}}\frac{2^{\frac{(N+8)}{2}(\gamma^{-1}+\gamma^{-2}+\cdots+\gamma^{1-k})}}{\gamma^{2(\gamma^{-1}+2\gamma^{-2}+\cdots+(k-1)\gamma^{1-k})}}\left(\fint_{Q_k}w^{m_k}\,dx\,dt\right)^\frac{1}{m_k}.
\end{split}
\end{equation}
Letting $k\to\infty$ in \eqref{sbigp}, the result follows for $\beta=2$. If $0<\beta<1$, by Young's inequality
\begin{equation}\label{slessp}
\begin{split}
\esssup_{U^-(\sigma' r)}\,w&\leq\left(\frac{C}{(\sigma-\sigma')^{\theta}}\right)^\frac{1}{2}\left(\fint_{U^{-}(\sigma r)}w^2\,dx\,dt\right)^\frac{1}{2}\\
&\leq\left(\frac{2-\beta}{4}\esssup_{U^{-}(\sigma r)}w\right)^\frac{2-\beta}{2}\left(\frac{4}{2-\beta}\right)^\frac{2-\beta}{2}
\left(\frac{C}{(\sigma-\sigma')^{\theta}}\right)^\frac{1}{2}\left(\fint_{U^-(\sigma r)}w^\beta\,dx\,dt\right)^\frac{1}{2}\\
&\leq\frac{1}{2}\esssup_{U^-(\sigma r)}w+\left((2-\beta)^{\beta-2}\frac{C}{(\sigma-\sigma')^{\theta}}\right)^\frac{1}{\beta}\left(\fint_{U^-(\sigma r)}w^\beta\,dx\,dt\right)^\frac{1}{\beta}\\
&\leq\frac{1}{2}\esssup_{U^-(\sigma r)}w+\left(\frac{C}{(\sigma-\sigma')^{\theta}}\right)^\frac{1}{\beta}\left(\fint_{U^-(\sigma r)}w^\beta\,dx\,dt\right)^\frac{1}{\beta}.
\end{split}
\end{equation}
The result follows by a similar iteration argument as in \cite[Lemma $5.1$]{Giaq}.
\end{proof} 

We obtain a reverse H\"older inequality for weak supersolutions of the problem \eqref{Problem}.

\begin{Lemma}\label{revHolderlemma}
Assume that $u$ is a weak supersolution of \eqref{Problem} such that $u\geq 0$ in $B_R(x_0)\times(t_0,t_0+r^2)\subset\Om\times(0,T)$. 
Let $0<r\leq 1$ be such that $r<\frac{R}{2}$, $d>0$ and $w=u+l$, with
\begin{equation}\label{l2}
l\geq\Big(\frac{r}{R}\Big)^2\mathrm{Tail}_{\infty}\big(u_-;x_0,R,t_0,t_0+r^2\big)+d,
\end{equation}
where $\mathrm{Tail}_{\infty}$ is defined by \eqref{loctail}. Let $\gamma=2-\frac{2}{\kappa}$, where $\kappa>2$ as given by \eqref{kappa}. Then there exist positive constants $C=C(N,s,\Lambda,q)$ and $\theta=\theta(\kappa)$ such that
\begin{equation}\label{revHolderineq}
\left(\fint_{U^+(\sigma'r)}w^q\,dx\,dt\right)^\frac{1}{q}
\leq\left(\frac{C}{(\sigma-\sigma')^\theta}\right)^\frac{1}{\overline{q}}\left(\fint_{U^+(\sigma r)}w^{\overline{q}}\,dx\,dt\right)^\frac{1}{\overline{q}},
\end{equation}
for all $\frac{1}{2}\leq\sigma'<\sigma\leq 1$ and $0<\overline{q}<q<\gamma$ and
$U^+(\eta r)=B_{\eta r}(x_0)\times(t_0,t_0+(\eta r)^2)$, $\frac{1}{2}\leq\eta\leq 1$.
\end{Lemma}

\begin{proof}
We divide the interval $(\sigma,\sigma')$ into $k$ parts by setting
$$
\sigma_0=\sigma,\sigma_k=\sigma',\sigma_j=\sigma-(\sigma-\sigma')\frac{1-\gamma^{-j}}{1-\gamma^{-k}},
\quad j=1,\dots,k-1.
$$
We shall choose $k$ below. Denote $r_j=\sigma_j r$, $B_j=B_{r_j}(x_0)$,
$\Gamma_j=(t_0,t_0+r_j^2)$ and $Q_j=U^+(r_j)=B_j\times\Gamma_j$.
We choose $\psi_j\in C_{c}^\infty(B_j)$ and $\eta_j\in C^\infty(\mathbb{R})$ such that
$0\leq\psi_j\leq 1$ in $B_j$, $\psi_j\equiv 1$ in $B_{j+1}$,
$0\leq\eta_j\leq 1$ in $\Gamma_j$, $\eta_j(t)=1$ for every $t\leq t_0+r_{j+1}^2$,
$\eta_j(t)=0$ for every $t\geq t_0+r_j^{2}$,
$\mathrm{dist}\,\big(\mathrm{supp}\,\psi,\mathbb{R}^N\setminus B_j\big)\geq 2^{-j-1}r$,
\[
|\nabla\psi_j|\leq\frac{8\gamma^j}{r(\sigma-\sigma')}
\quad\text{and}\quad
\left|\frac{\partial\eta_j}{\partial t}\right|\leq 8\left(\frac{\gamma^j}{r(\sigma-\sigma')}\right)^2\text{ in }Q_j.
\]
Let $0<\epsilon<1$ and $\alpha=1-\epsilon$. Using Lemma \ref{c.omega_sobo} along with H\"older's inequality with exponents $t=\frac{\kappa}{\kappa-2}$ and $t'=\frac{\kappa}{2}$, for some positive constant $C=C(N)$, we have
\begin{equation}\label{revSob}
\begin{split}
&\fint_{Q_{j+1}}w^{\gamma\alpha}\,dx\,dt
=\fint_{\Gamma_{j+1}}\fint_{B_{j+1}}w^{\gamma \alpha}\,dx dt=\fint_{\Gamma_{j+1}}\fint_{B_{j+1}}w^{\frac{\alpha}{t}+\frac{\alpha\kappa}{2t'}}\,dx\,dt\\
&\leq\fint_{\Gamma_{j+1}}\frac{1}{|B_{j+1}|}\left(\int_{B_{j+1}}w^\alpha\,dx\right)^\frac{1}{t}\left(\int_{B_{j+1}}w^\frac{\alpha\kappa}{2}\,dx\right)^\frac{1}{t'}dt\\
&\leq\fint_{\Gamma_{j+1}}\frac{1}{|B_{j+1}|}\left(\esssup_{\Gamma_{j+1}}\int_{B_{j}}w^\alpha\psi_j^{2}\,dx \right)^\frac{1}{t}\left(\int_{B_{j}}w^\frac{\alpha\kappa}{2}\psi_j^\frac{(\alpha+2)\kappa}{2}\eta_j(t)^\frac{\kappa}{2}\,dx\right)^\frac{1}{t'}dt\\
&=\frac{|\Gamma_j||B_j|}{|\Gamma_{j+1}||B_{j+1}|}\left(\esssup_{\Gamma_{j+1}}\fint_{B_{j}}w^\alpha\psi_j^{2}\,dx \right)^\frac{1}{t}\fint_{\Gamma_{j}}\left(\fint_{B_{j}}w^\frac{\alpha\kappa}{2}\psi_j^\frac{(\alpha+2)\kappa}{2}\eta_j(t)^\frac{\kappa}{2}\,dx\right)^\frac{1}{t'}dt\\
&\leq C\left(\esssup_{\Gamma_{j+1}}\fint_{B_{j}}w^\alpha\psi_j^{2}\,dx\right)^\frac{1}{t}r^2\fint_{\Gamma_j}\fint_{B_j}\Big|\nabla\Big(w^\frac{\alpha}{2}\psi_j^\frac{\alpha+2}{2}\eta_j^\frac{1}{2}\Big)\Big|^2\,dx\,dt\\
&=\frac{C}{r_j^{\gamma N}}\left(\esssup_{\Gamma_{j+1}}\int_{B_{j}}w^\alpha\psi_j^{2}\,dx \right)^\frac{1}{t}
\int_{\Gamma_j}\int_{B_j}\left(\big|\nabla\big(w^\frac{\alpha}{2}\big)\big|^2\psi_j(x)^{\alpha+2}\eta_j(t)+\alpha^2 w^\alpha |\nabla\psi_j|^2\right)\,dx\,dt\\
&\leq\frac{C}{r_j^{\gamma N}}\left(\esssup_{\Gamma_{j+1}}\int_{B_{j}}w^\alpha\psi_j^{2}\,dx \right)^\frac{1}{t}
\int_{\Gamma_j}\int_{B_j}\left(\big|\nabla\big(w^\frac{\alpha}{2}\big)\big|^2\psi_j(x)^{2}\eta_j(t)+\alpha^2 w^\alpha |\nabla\psi_j|^2\right)\,dx\,dt.
\end{split}
\end{equation}
Let
\[
I=\esssup_{\Gamma_{j+1}}\int_{B_{j}}w^\alpha\psi_j^{2}\,dx\,dt,\quad 
J=\int_{\Gamma_j}\int_{B_j}\big|\nabla w^\frac{\alpha}{2}\big|^2\psi_j^{2}\eta_j\,dx\,dt
\]
and
\[
K=\alpha^2\int_{\Gamma_j}\int_{B_j}w^\alpha |\nabla\psi_j|^2\,dx\,dt.
\]
Setting $r=r_j, \tau_1=t_0, \tau_2=t_0+r_{j+1}^2$ and $\tau=r_j^2-r_{j+1}^2$ in Lemma \ref{inveng}, for some constant $C=C(\Lambda)>0$, we have
\begin{equation}\label{IJexprev}
\begin{split}
I,J&\leq \frac{C}{\epsilon^2}\Bigg[\int_{\Gamma_j}\int_{B_j}\int_{B_j}\big(\psi_j(x)-\psi_j(y)\big)^2\big(w(x,t)^\alpha+w(y,t)^\alpha\big)\eta_j(t)\,d\mu\,dt\\
&+2\esssup_{x\in\mathrm{supp}\psi_j}\int_{\mathbb{R}^N\setminus B_j}\frac{dy}{|x-y|^{N+2s}}\int_{\Gamma_j}\int_{B_j}w(x,t)^\alpha\psi_j(x)^{2}\eta_j(t)\,dx\,dt\\
&+\frac{2}{l}\esssup_{t\in\Gamma_j,\,x\in\mathrm{supp}\psi_j}\int_{\mathbb{R}^N\setminus B_R(x_0)}\frac{u_-(y,t)}{|x-y|^{N+2s}}\,dy\int_{\Gamma_j}\int_{B_j}w(x,t)^\alpha\psi_j(x)^{2}\eta_j(t)\,dx\,dt\\
&+\int_{\Gamma_j}\int_{B_j}w^{\alpha}|\nabla\psi_j|^2\eta_j\,dx\,dt
+\int_{\Gamma_j}\int_{B_j}w(x,t)^\alpha\psi_j(x)^2|\partial_{t}\eta_j(t)|\,dx\,dt\Bigg]\\
&=I_1+I_2+I_3+I_4+I_5.
\end{split}
\end{equation}
Proceeding as in the proof of Lemma \ref{invlemma}, we obtain
\begin{equation}\label{estI_1rev}
\begin{split}
I_1&=\frac{C}{\epsilon^2}\int_{\Gamma_j}\int_{B_j}\int_{B_j}\big(\psi_j(x)-\psi_j(y)\big)^2\big(w(x,t)^\alpha+w(y,t)^\alpha\big)\eta_j(t)\,d\mu\,dt\\
&\leq\frac{C\gamma^{2j}}{\epsilon^2 r_j^{2}(\sigma-\sigma')^2}\int_{\Gamma_j}\int_{B_j}w(x,t)^{\alpha}\,dx\,dt,
\end{split}
\end{equation}
and
\begin{equation}\label{estI_2rev}
\begin{split}
I_2&=\frac{C}{\epsilon^2}\esssup_{x\in\mathrm{supp}\,\psi_j}\int_{\mathbb{R}^N\setminus B_j}\frac{dy}{|x-y|^{N+2s}}\int_{\Gamma_j}\int_{B_j}w(x,t)^{\alpha}\psi_j(x)^{2}\eta_j(t)\,dx\,dt\\
&\leq\frac{C2^{j(N+2s)}}{\epsilon^2 r_j^{2}(\sigma-\sigma')^2}\int_{\Gamma_j}\int_{B_j}w(x,t)^{\alpha}\,dx\,dt,
\end{split}
\end{equation}
for some constant $C=C(N,s,\Lambda)>0$. Again arguing similarly as in the estimate of $I_3$ in the proof of Lemma \ref{invlemma} and noting \eqref{l2}, we have
\begin{equation}\label{estI_3rev}
\begin{split}
I_3&=\frac{C}{l\epsilon^2}\esssup_{t\in\Gamma_j,\,x\in\mathrm{supp}\psi_j}\int_{\mathbb{R}^N\setminus B_R(x_0)}\frac{u_-(y,t)}{|x-y|^{N+2s}}\,dy\int_{\Gamma_j}\int_{B_j}w(x,t)^{\alpha}\psi_j(x,t)^{2}\eta_j(t)\,dx\,dt\\
&\leq\frac{C\gamma^{2j}}{\epsilon^2 r_j^2(\sigma-\sigma')^2}\int_{\Gamma_j}\int_{B_j}w(x,t)^{\alpha}\,dx\,dt,
\end{split}
\end{equation}
\begin{equation}\label{estI_4rev}
I_4=\frac{C}{\epsilon^2}\int_{\Gamma_j}\int_{B_j}\int_{B_j}|\nabla\psi_j|^2 w(x,t)^{\alpha}\psi_j(x)^{\alpha}\eta_j(t)\,dx\,dt
\leq\frac{C\gamma^{2j}}{\epsilon^2 r_j^{2}(\sigma-\sigma')^2}\int_{B_j}w(x,t)^{\alpha}\,dx dt,
\end{equation}
\begin{equation}\label{estI_5rev}
I_5=\frac{C}{\epsilon^2}\int_{\Gamma_j}\int_{B_j}w(x,t)^{\alpha}\psi_j(x)^{2}|\partial_t\eta_j(t)|\,dx\,dt\\
\leq\frac{C\gamma^{2j}}{\epsilon^2 r_j^{2}(\sigma-\sigma')^2}\int_{\Gamma_j}\int_{B_j}w(x,t)^{\alpha}\,dx\,dt,
\end{equation}
for some positive constant $C=C(\Lambda,N,s)$. Plugging the estimates \eqref{estI_1rev}\, \eqref{estI_2rev}, \eqref{estI_3rev}, \eqref{estI_4rev} and \eqref{estI_5rev} in \eqref{IJexprev}, since $\gamma<2$, we obtain
\begin{equation}\label{estIJrev}
\begin{split}
I,J&\leq\frac{C2^{j(N+4)}}{\epsilon^2 r_j^{2}(\sigma-\sigma')^2}\int_{\Gamma_j}\int_{B_j}w(x,t)^{\alpha}\,dx\,dt.
\end{split}
\end{equation}
for some positive constant $C=C(N,s,\Lambda)$.
Using the properties of $\psi_j$, we obtain
\begin{equation}\label{estKrev}
\begin{split}
K&\leq\frac{C2^{2j}}{\epsilon^2 r_j^{2}(\sigma-\sigma')^2}\int_{\Gamma_j}\int_{B_j}w(x,t)^{\alpha}\,dx\,dt.
\end{split}
\end{equation}
for some positive constant $C=C(N,s,\Lambda)$. As in the proof of Lemma \ref{invlemma}, employing the estimates \eqref{estIJrev} and \eqref{estKrev} in \eqref{revSob}, for $\gamma=2-\frac{2}{\kappa}$, we have
\begin{equation}\label{revMoser}
\fint_{Q_{j+1}}w^{\gamma\alpha}\,dx\,dt
\leq C\left(\frac{2^{j(N+4)}}{(\sigma-\sigma')^2}\fint_{Q_j}w^\alpha\,dx\,dt\right)^\gamma,
\end{equation}
for some positive constant $C=C(N,s,\Lambda)$. Note that $C$ is independent of $\epsilon$ as long as $\alpha$ is away from $1$.
We use the Moser iteration technique to conclude the result. Fix $q$ and $\overline{q}$ such that $q>\overline{q}$ and $k$ such that $\overline{q}\gamma^{k-1}\leq q\leq\overline{q}\gamma^k$. 
Let $t_0$ be such that $t_0\leq\overline{q}$ and $q=\gamma^k t_0$. 
Let $t_j=\gamma^j t_0$, $j=0,1,\cdots,k$. 
By iterating \eqref{revMoser} and H\"older's inequality, we arrive at
\begin{equation}\label{revMoserfinal}
\begin{split}
\left(\fint_{Q_k}w^q\,dx\,dt\right)^\frac{1}{q}
&\leq\left(\frac{C^*}{(\sigma-\sigma')^\beta}\fint_{Q_0}w^{t_0}\,dx\,dt\right)^\frac{1}{t_0},\\
&\leq\left(\frac{C^*}{(\sigma-\sigma')^\beta}\right)^\frac{1}{t_0}\left(\fint_{Q_0}w^{\overline{q}}\,dx\,dt\right)^\frac{1}{\overline{q}},
\end{split}
\end{equation}
where
$$
C^*=2^{(N+4)(\gamma^{-1}+2\gamma^{-2}+\cdots+(k-1)\gamma^{1-k})}C^{\gamma^{-1}+\gamma^{-2}+\cdots+\gamma^{-k}}
$$
and
$$
\beta=2\big(1+\gamma^{-1}+\gamma^{-2}+\cdots+\gamma^{1-k}\big)=\frac{2\gamma}{\gamma-1}\big(1-\gamma^{-k}\big).
$$
Note that, due to the singularity of $\epsilon$ at $0$ in the estimates \eqref{estIJrev} and \eqref{estKrev}, the constant $C$ in \eqref{revMoserfinal} depends on $q$. It is easy to observe that $C^*$ and $\beta$ are uniformly bounded over $k$. Now, since $\overline{q}\gamma^{k-1}\leq t_0\gamma^k$, we have $t_0\geq\frac{\overline{q}}{\gamma}$. Hence, the result follows from \eqref{revMoserfinal}, with $\theta=\frac{2\gamma^2}{\gamma-1}$.
\end{proof}

Next, we prove the following logarithmic estimate for weak supersolutions of \eqref{Problem}.
\begin{Lemma}\label{Logestimatelemma}
Assume that $u$ is a weak supersolution of the problem \eqref{Problem} such that $u\geq 0$ in $B_R(x_0)\times(t_0-r^2,t_0+r^2)\subset\Om\times(0,T)$. 
Let $0<r\leq 1$ be such that $r<\frac{R}{2}$, $d>0$ and $v=u+l$, with
$$
l=\Big(\frac{r}{R}\Big)^2\mathrm{Tail}_{\infty}\big(u_-;x_0,R,t_0-r^2,t_0+r^2\big)+d,
$$
where $\mathrm{Tail}_{\infty}$ is defined by \eqref{loctail}. Then there exists a constant $C=C(N,s,\Lambda)>0$ such that
\begin{equation}\label{log1}
\big|U^{+}(r)\cap\{\log v<-\lambda-b\}\big|\leq \frac{Cr^{N+2}}{\lambda},
\end{equation}
where $U^+(r)=B_r(x_0)\times(t_0,t_0+r^2)$.
Moreover, there exists a constant $C=C(N,s,\Lambda)>0$ such that
\begin{equation}\label{log2}
\big|U^{-}(r)\cap\{\log v>\lambda-b\}\big|\leq \frac{Cr^{N+2}}{\lambda},
\end{equation}
where $U^-(r)=B_r(x_0)\times(t_0-r^2,t_0)$.
Here
\begin{equation}\label{b}
b=b\big(v(\cdot,t_0)\big)=-\frac{\int_{B_{\frac{3r}{2}}(x_0)}\log v(x,t_0)\psi(x)^2\,dx}{\int_{B_{\frac{3r}{2}}(x_0)}\psi(x)^2\,dx},
\end{equation}
where $\psi\in C_c^{\infty}\big(B_{\frac{3r}{2}}(x_0)\big)$ is a nonnegative, radially decreasing function such that $0\leq\psi\leq 1$ in $B_\frac{3r}{2}(x_0)$, $\psi\equiv 1$ \text{ in }$B_r(x_0)$, $|\nabla\psi|\leq\frac{C}{r}$ in $B_\frac{3r}{2}(x_0)$, for some constant $C>0$ independent of $r$.
\end{Lemma}
\begin{proof}
We only prove the estimate \eqref{log1}, since the proof of \eqref{log2} follows similarly. Without loss of generality, we assume that $x_0=0$ and denote by $B_r=B_r(0)$, $B_\frac{3r}{2}=B_\frac{3r}{2}(0)$. Since $v$ is a weak supersolution of \eqref{Problem}, choosing $\phi(x,t)=\psi(x)^2 v(x,t)^{-1}$ as a test function in \eqref{wksol} (which can again be justified as in the proof of Lemma \ref{inveng}), we get
\begin{equation}\label{log1eqn}
I_1+I_2+2I_3+I_4\geq 0,
\end{equation}
where for any $t_0\leq t_1<t_2\leq t_0+r^{2}$, we have
\begin{equation}\label{logI1}
I_1=\int_{B_{\frac{3r}{2}}}\log v(x,t)\psi(x)^2\,dx\Big|_{t=t_1}^{t_2}.
\end{equation}
Following the arguments as in the proof of \cite[Lemma 1.3]{Kuusilocal}, for some constant $C=C(N,s,\Lambda)>0$, we obtain 
\begin{equation}\label{logI2}
\begin{split}
I_2&=\int_{t_1}^{t_2}\int_{B_{\frac{3r}{2}}}\int_{B_{\frac{3r}{2}}}\mathcal{A}\big(v(x,y,t)\big)\big(\psi(x)^2 v(x,t)^{-1}-\psi(y)^2 v(y,t)^{-1}\big)\,d\mu\, dt\\
&\leq -\frac{1}{C}\int_{t_1}^{t_2}\int_{B_{\frac{3r}{2}}}\int_{B_{\frac{3r}{2}}}K(x,y,t)|\log v(x,t)-\log v(y,t)|^2 \psi(y)^2\,dx\, dy\, dt\\
&\qquad+C\int_{t_1}^{t_2}\int_{B_{\frac{3r}{2}}}\int_{B_{\frac{3r}{2}}}K(x,y,t)|\psi(x)-\psi(y)|^2\,dx\, dy\, dt\\
&\leq -\frac{1}{C}\int_{t_1}^{t_2}\int_{B_{\frac{3r}{2}}}\int_{B_{\frac{3r}{2}}}K(x,y,t)|\log v(x,t)-\log v(y,t)|^2 \psi(y)^2\,dx\, dy\, dt\\
&\qquad+C(t_2-t_1)r^{N-2}, 
\end{split}
\end{equation}
where the last inequality is obtained using the properties of $\psi$ and the fact that $0<r\leq 1$. Noting that $v\geq l$ in $B_R(x_0)\times(t_0,t_0+r^2)$ and arguing as in the proof of the estimate \eqref{estI2}, we have
\begin{equation}\label{logI3}
\begin{split}
I_3&=\int_{t_1}^{t_2}\int_{\mathbb{R}^N\setminus B_{\frac{3r}{2}}}\int_{B_{\frac{3r}{2}}}\mathcal{A}\big(v(x,y,t)\big)\psi(x)^2 v(x,t)^{-1}\,d\mu\, dt\\
&\leq C(t_2-t_1)r^{N-2}+\frac{2\Lambda}{l}\int_{t_1}^{t_2}\int_{\mathbb{R}^N\setminus B_R(x_0)}\int_{B_\frac{3r}{2}(x_0)}\frac{u_-(y,t)}{|x-y|^{N+2s}}\,dx\,dy\,dt\\
&\leq C(t_2-t_1)r^{N-2}+\frac{2\Lambda}{l}(t_2-t_1)r^N R^{-2}\mathrm{Tail}_{\infty}\big(u_-;x_0,R,t_0,t_0+r^2\big)\\
&\leq C(t_2-t_1)r^{N-2},
\end{split}
\end{equation}
for some constant $C=C(N,s,\Lambda)>0$.
Using Young's inequality, for some constant $C=C(N,s,\Lambda)>0$, we obtain
\begin{equation}\label{logI4}
\begin{split}
I_4&=\int_{t_1}^{t_2}\int_{B_{\frac{3r}{2}}}\nabla v\nabla\big(\psi^2 v^{-1}\big)\,dx\,dt\\
&\leq-\frac{1}{2}\int_{t_1}^{t_2}\int_{B_{\frac{3r}{2}}}|\nabla\log v|^2\psi^2\,dx\,dt
+C\int_{t_1}^{t_2}\int_{B_{\frac{3r}{2}}}|\nabla\psi|^2\,dx\,dt\\
&\leq-\frac{1}{2}\int_{t_1}^{t_2}\int_{B_{\frac{3r}{2}}}|\nabla\log \,v|^2\psi^2\,dx\,dt+C(t_2-t_1)r^{N-2}. 
\end{split}
\end{equation}
Therefore using the estimates \eqref{logI1}, \eqref{logI2}, \eqref{logI3} and \eqref{logI4} in \eqref{log1eqn}, we obtain
\begin{equation}\label{log2old}
\begin{split}
&\frac{1}{C}\int_{t_1}^{t_2}\int_{B_{\frac{3r}{2}}}\int_{B_{\frac{3r}{2}}}K(x,y,t)|\log u(x,t)-\log u(y,t)|^2 \psi(y)^2\,dx\, dy\, dt\\
&\quad+\int_{t_1}^{t_2}\int_{B_{\frac{3r}{2}}}|\nabla\log v|^2\psi^2\,dx\,dt
-\int_{B_{\frac{3r}{2}}}\log v(x,t)\psi(x)^2\,dx\Big|_{t=t_1}^{t_2}\leq C(t_2-t_1)r^{N-2},
\end{split}
\end{equation}
which gives the estimate,
\begin{equation}\label{log2eqn}
\int_{t_1}^{t_2}\int_{B_{\frac{3r}{2}}}|\nabla\log v|^2\psi^2\,dx\,dt-\int_{B_{\frac{3r}{2}}}\log v(x,t)\psi(x)^2\,dx\Big|_{t=t_1}^{t_2}\leq C(t_2-t_1)r^{N-2}.
\end{equation}
Let $w(x,t)=-\log v(x,t)$ and
$$
W(t)=\frac{\int_{B_{\frac{3r}{2}}}w(x,t)\psi(x)^2\,dx}{\int_{B_{\frac{3r}{2}}}\psi(x)^2\,dx}.
$$
Since $0\leq\psi\leq 1$ in $B_{\frac{3r}{2}}$ and $\psi\equiv 1$ in $B_r$, we obtain $\int_{B_{\frac{3r}{2}}}\psi(x)^2\,dx\approx r^N$. Then by Lemma \ref{wgtPoin}, for some positive constant $C_1>0$ (independent of $r$), we obtain
\begin{equation}\label{wgtPoinapp}
\frac{1}{r^2}\fint_{B_\frac{3r}{2}}|w-W(t)|^2\psi^2\,dx\leq C_1\frac{\int_{B_\frac{3r}{2}}|\nabla w|^2\psi^2\,dx}{\int_{B_\frac{3r}{2}}\psi^2\,dx}.
\end{equation}
By dividing through by $\int_{B_{\frac{3r}{2}}}\psi^2\,dx$ on both sides of \eqref{log2eqn} and using \eqref{wgtPoinapp} together with the fact that $\psi\equiv 1$ in $B_r$, we get
$$
W(t_2)-W(t_1)+\frac{1}{C_1 r^{2}}\int_{t_1}^{t_2}\fint_{B_r}|w(x,t)-W(t)|^2\,dx\, dt\leq C(t_2-t_1)r^{-2}.
$$
Let $A_1=C_1$, $A_2=C$, $\overline{w}(x,t)=w(x,t)-A_2 r^{-2}(t-t_1)$ and $\overline{W}(t)=W(t)-A_2 r^{-2}(t-t_1)$.
Then $w(x,t)-W(t)=\overline{w}(x,t)-\overline{W}(t)$.
Hence we get
\begin{equation}\label{monotone1}
\overline{W}(t_2)-\overline{W}(t_1)+\frac{1}{A_1 r^{N+2}}\int_{t_1}^{t_2}\int_{B_r}|\overline{w}(x,t)-{\overline{W}(t)}|^2\, dx\, dt\leq 0.
\end{equation}
This shows that $\overline{W}(t)$ is a monotone decreasing function in $(t_1,t_2)$. Hence  $\overline{W}(t)$ is differentiable almost everywhere with respect to $t$. Hence, from \eqref{monotone1} for almost every $t$ such that $t_1<t<t_2$, we obtain
\begin{equation}\label{notime}
\overline{W}'(t)+\frac{1}{A_1 r^{N+2}}\int_{B_r}\big|\overline{w}(x,t)-\overline{W}(t)\big|^2\,dx\leq 0.
\end{equation}
Setting $t_1=t_0$ and $t_2=t_0+r^2$, we have $\overline{W}(t_0)=W(t_0)$ and we denote by $b\big(v(\cdot,t_0)\big)=\overline{W}(t_0)$. 
Let
$$
\Omega_{t}^{+}(\lambda)=\big\{x\in B_r:\overline{w}(x,t)>b+\lambda\big\}.
$$
For every $t\geq t_0$, we have $\overline{W}(t)\leq \overline{W}(t_0)=b$. Thus $x\in\Omega_{t}^{+}(\lambda)$ gives
\[
\overline{w}(t,x)-\overline{W}(t)\geq b+\lambda-\overline{W}(t)\geq b+\lambda-\overline{W}(t_0)=\lambda>0.
\]
Hence from \eqref{notime}, we have 
$$
\overline{W}'(t)+\frac{|\Omega_{t}^{+}(\lambda)|}{A_1 r^{N+2}}\big(b+\lambda-\overline{W}(t)\big)^2\leq 0.
$$
Therefore, we have
$$
|\Omega_{t}^{+}(\lambda)|\leq-A_1 r^{N+2}\partial_{t}\big(b+\lambda-\overline{W}(t)\big)^{-1}.
$$
Integrating over $t_0$ to $t_0+r^{2}$, we obtain
$$
\big|\{(x,t)\in B_r\times(t_0,t_0+r^{2}):\overline{w}(x,t)>b+\lambda\}\big|\leq-{A_1 r^{N+2}}\int_{t_0}^{t_0+r^{2}}\partial_{t}\big(b+\lambda-\overline{W}(t)\big)^{-1}\,dt,
$$
which gives
\begin{equation}\label{logsubpositivetime}
\big|\{(x,t)\in B_r\times(t_0,t_0+r^{2}):\log v(x,t)+A_2 r^{-2}(t-t_0)<-\lambda-b\}\big|\leq{A_1}\frac{r^{N+2}}{\lambda}.
\end{equation}
Finally we obtain
\begin{equation}\label{logpostivetime}
\big|\{(x,t)\in B_r\times(t_0,t_0+r^{2}):\log v(x,t)<-\lambda-b\}\big|\leq A+B,
\end{equation}
where using \eqref{logsubpositivetime}, for some positive constant $C=C(N,s,\Lambda)$, we obtain
$$
A=\big|\{(x,t)\in B_r\times(t_0,t_0+r^{2}):\log v(x,t)+A_2 r^{-2}(t-t_0)<-\tfrac{\lambda}{2}-b\}\big|
\leq \frac{Cr^{N+2}}{\lambda}, 
$$
and
$$
B=\big|\{(x,t)\in B_r\times(t_0,t_0+r^{2}):A_2 r^{-2}(t-t_0)>\tfrac{\lambda}{2}\}\big|
\leq\left(1-\frac{\lambda}{2A_2}\right)r^{N+2}.  
$$
If $\frac{\lambda}{2 A_2}<1$, then 
$$B\leq\left(1-\frac{\lambda}{2A_2}\right)r^{N+2}<r^{N+2}<\left(\frac{2 A_2}{\lambda}\right)r^{N+2}.$$
If $\frac{\lambda}{2 A_2}\geq 1$, then $B\leq 0$. Hence in any case we have
$$
B\leq\frac{Cr^{N+2}}{\lambda},
$$
for some positive constant $C=C(N,s,\Lambda)$. Inserting the above estimates on $A$ and $B$ into \eqref{logpostivetime}, we obtain
\begin{equation*}\label{logpostivetimefinal}
\big|\{(x,t)\in B_r\times(t_0,t_0+r^{2}):\log v(x,t)<-\lambda-b\}\big|\leq\frac{Cr^{N+2}}{\lambda},
\end{equation*}
for some positive constant $C=C(N,s,\Lambda)$, which proves the estimate \eqref{log1}.
\end{proof}

\section{Proof of the main result}
\textbf{Proof of Theorem \ref{mainthm}:} For $\frac{1}{2}\leq\eta\leq 1$, denote
$$
V^+(\eta r)=B_{\eta r}(x_0)\times(t_0+r^2-(\eta r)^2,t_0+r^2)
\quad\text{and}\quad
V^-(\eta r)=B_{\eta r}\times(t_0-r^2,t_0-r^2+(\eta r)^2).
$$
Then $V^+(r)=B_r(x_0)\times(t_0,t_0+r^2)=U^+(r)$ and $V^-(r)=B_r(x_0)\times(t_0-r^2,t_0)=U^-(r)$.
Observe that $V^-$ and $V^+$ are sequences of nondecreasing cylinders over the interval $[\frac{1}{2},1]$. Let $u$ be as given in the hypothesis and for any $d>0$, assume that
$v=u+l$,
where
$$
l=\Big(\frac{r}{R}\Big)^2\mathrm{Tail}_{\infty}\big(u_-;x_0,R,t_0-r^2,t_0+r^2\big)+d,
$$
where $\mathrm{Tail}_{\infty}$ is defined by \eqref{loctail}. 
We denote $w_1=e^{-b}v^{-1}$ and $w_2=e^{b}v$, where $b$ is given by \eqref{b}. 
By applying [\eqref{log1}, Lemma \ref{Logestimatelemma}], for any $\lambda>0$ with some constant $C=C(N,s,\Lambda)>0$, we have
\begin{equation}\label{logapp1}
\left|V^+(r)\cap\{\log w_1>\lambda\}\right|\leq\frac{C|V^+(\frac{r}{2})|}{\lambda}.
\end{equation}
Moreover, using Lemma \ref{invlemma}, for any $0<\beta<1$, there exists some constant $C=C(N,s,\Lambda)>0$ such that
\begin{equation}\label{invlemapp}
\esssup_{V^+(\sigma' r)}\,w_1\leq\left(\frac{C}{(\sigma-\sigma')^{\theta}}\fint_{V^+(\sigma r)}w_1^{\beta}\,dx\,dt\right)^\frac{1}{\beta},
\end{equation}
for $\frac{1}{2}\leq\sigma'<\sigma\leq 1$. Therefore, using \eqref{logapp1} and \eqref{invlemapp} in Lemma \ref{Bombieri}, we have
\begin{equation}\label{Bombieriapp1}
\esssup_{V^+(\frac{r}{2})}\,w_1\leq C_1,
\end{equation}
for some constant $C_1=C_1(\theta,N,s,\Lambda)>0$.
By \eqref{log2} in Lemma \ref{Logestimatelemma} there exists a constant $C=C(N,s,\Lambda)>0$ such that
\begin{equation}\label{logapp2}
\left|V^-(r)\cap\{\log w_2>\lambda\}\right|\leq\frac{C|V^-(\frac{r}{2})|}{\lambda}
\end{equation}
for every $\lambda>0$.
By Lemma \ref{revHolderlemma} there exists a constant $C=C(N,s,\Lambda,q)>0$ such that
\begin{equation}\label{revHolderapp}
\left(\fint_{V^-(\sigma' r)}\,w_2^q\,dx\,dt\right)^\frac{1}{q}
\leq\left(\frac{C}{(\sigma-\sigma')^{\theta}}\fint_{V^-(\sigma r)}w_2^{\overline{q}}\,dx\,dt\right)^\frac{1}{\overline{q}},
\end{equation}
for $\frac{1}{2}\leq\sigma'<\sigma\leq 1$ and $0<\overline{q}<q<2-\frac{2}{\kappa}$, where $\kappa$ is given by \eqref{kappa}. Therefore, using \eqref{logapp2} and \eqref{revHolderapp} in Lemma \ref{Bombieri}, we have
\begin{equation}\label{Bombieriapp2}
\left(\fint_{V^-(\frac{r}{2})}w_2^q\,dx dt\right)^\frac{1}{q}\leq C_2,
\end{equation}
for some constant $C_2=C_2(\theta,N,s,\Lambda,q)>0$. Multiplying \eqref{Bombieriapp1} and \eqref{Bombieriapp2}, for any $0<q<2-\frac{2}{\kappa}$, we have 
\begin{equation}\label{fwkest}
\left(\fint_{V^-(\frac{r}{2})}v^q\,dx\,dt\right)^\frac{1}{q}
\leq C_1 C_2\essinf_{V^+(\frac{r}{2})}\,v\leq C_1 C_2\Big( \essinf_{V^+(\frac{r}{2})}\,u+l+d\Big).
\end{equation}
Since $d>0$ is arbitrary, the estimate \eqref{thm1ine} follows from \eqref{fwkest}. \qed

\end{document}